\documentclass[a4paper,10pt]{article}
\usepackage{amsmath,amscd,amsfonts,amssymb,epsf,latexsym}

\makeatletter

\long\def\@makefnt#1{\parindent 1em\noindent
            \hb@xt@1.8em{\hss\@textsuperscript{}}#1}
\long\def\@ftntext#1{\insert\footins{%
    \reset@font\footnotesize
    \interlinepenalty\interfootnotelinepenalty
    \splittopskip\footnotesep
    \splitmaxdepth \dp\strutbox \floatingpenalty \@MM
    \hsize\columnwidth \@parboxrestore
    \color@begingroup
      \@makefnt{%
        \rule\z@\footnotesep\ignorespaces#1\@finalstrut\strutbox}%
    \color@endgroup}}%
\def\subjclass#1{%
  \@ftntext{2010 {\itshape Mathematics Subject Classification.}\enspace #1.}}
\def\keywords#1{%
  \@ftntext{{\itshape Key words and phrases.}\enspace #1.}}
\makeatother

\def\A{{\mathbb A}}
\def\B{{\mathbb B}}

\def\C{{\mathbb C}}
\def\D{{\mathbb D}}

\def\X{{\mathbb X}}

\def\E{{\mathbb E}}
\def\F{{\mathbb F}}
\def\G{{\mathbb G}}
\def\H{{\mathbb H}}

\def\AB{ {\mathbb A}\moins {\mathbb B}}

\def\cJ{{\mathcal J}}
\def\cP{{\mathcal P}}
\def\cT{{\mathcal T}}

\def\card{\mathop{\rm card}}
\def\Aut{\mathop{\rm Aut}}

\def\moins{\raise 1pt\hbox{{$\scriptstyle -$}}}
\def\plus{\raise 1pt\hbox{{$\scriptstyle +$}} }

\def\phi{\varphi}

\newtheorem{theorem}{Theorem}
\newtheorem{proposition}[theorem]{Proposition}
\newtheorem{lemma}[theorem]{Lemma}
\newtheorem{corollary}[theorem]{Corollary}

\newtheorem{definition}[theorem]{Definition}

\newtheorem{convention}[theorem]{Convention}
\newtheorem{example}[theorem]{Example}
\newtheorem{note}[theorem]{Note}
\newtheorem{notation}[theorem]{Notation}

\def\proof{\noindent{\bf Proof.\ }}

\def\qed{~\hbox{$\Box$}}

\def\Aut{\mathop{\rm Aut}}

\def\codim{\mathop{\rm codim}}

\def\rank{\mathop{\rm rank}}
\def\dim{\mathop{\rm dim}}
\def\Ker{\mathop{\rm Ker}}

\begin{document}

\title{\bf On Schur function expansions of Thom polynomials}

\author{\"Ozer \"Ozt\"urk \\
\small Department of Mathematics of Mimar Sinan Fine Arts University\\
\small \c C\i ra\u gan C., \c Ci\u gdem S., No. 1\\
\small 34349, Be\c sikta\c s, Istanbul, Turkey\\
\small ozer.ozturk@msgsu.edu.tr \and
Piotr Pragacz\thanks{Research supported by the MNiSW grant N N201 608040, 
and the Japanese JSPS Grant-in-Aid for Scientific Research (B) 22340007.}\\
\small Institute of Mathematics of Polish Academy of Sciences\\
\small \'Sniadeckich 8, 00-956 Warszawa, Poland\\
\small P.Pragacz@impan.pl}



\date{}

\maketitle

%\tableofcontents

\rightline{\it \small Why should we expect a city to cure us of our spiritual pains? Perhaps} 
\rightline{\it \small because we cannot help, loving our city like a family. But we still have}
\rightline{\it \small to decide which part of the city we love and invent the reasons why.}
\vskip5pt
\rightline{\small Orhan Pamuk, Istanbul: memories and the city.}

\vskip40pt

\begin{abstract}
We discuss computations of the Thom polynomials of singularity classes of maps
in the basis of Schur functions. We survey the known results about the bound
on the length and a rectangle containment for partitions appearing in such Schur 
function expansions. We describe several recursions for the coefficients. For some singularities, 
we give old and new computations of their Thom polynomials.
\end{abstract}

\section{Introduction}\label{intro}

A prototype of the formulas considered in the present paper, is the following classical result.
Let $f: M\to N$ be a holomorphic, surjective map of compact Riemann surfaces. For $x\in M$, we set
$$
e_x:=\hbox{number of branches of} \ f  \ \hbox{at} \ x.
$$  
Then the ramification divisor of $f$ is equal to
$$
\sum (e_x-1) x\,.
$$ 
The {\it Riemann-Hurwitz formula} asserts that
\begin{equation}\label{rh}
\sum_{x\in M} (e_x-1)=2g(M)-2-\deg(f)\bigl(2g(N)-2\bigr)\,.
\end{equation}
(See, e.g., \cite{Kl}.) The right-hand side of Eq.(\ref{rh}) can be rewritten as
$$
f^*c_1(N)-c_1(M)\,,
$$ 
and gives us the Thom polynomial of the singularity $A_1$ of maps between curves.

\smallskip

In general, according to the monograph \cite{AVGL}, the global behavior of singularities 
of maps $f: M\to N$ of complex analytic manifolds, 
is governed by their {\it Thom polynomials}. Knowing the Thom polynomial of 
a singularity class $\Sigma$, one can compute the cohomology class represented by the 
$\Sigma$-points of a map $f$. We shall recall the definition of a Thom polynomial in Section \ref{Thom}. 

The term ``Thom polynomial'' has nowadays rather wide meaning. In the present paper, however,
it will mean a classical Thom polynomial of the singularity classes of maps (cf. \cite{T}). We shall
work here with {\it complex} manifolds\footnote{A manifold here is always {\it smooth}.}.

An explicit\footnote{Even the word ``explicit'' has different meanings for different authors working on Thom polynomials.} 
computation of a Thom polynomial is usually a difficult task.
At first, the computations of Thom polynomials were performed in the
basis of monomials in the Chern classes. But around 2004, two papers: \cite{FK}
and \cite{PArx} appeared independently, with computations of some Thom polynomials
in the basis of {\it Schur functions}. (The two papers concerned different singularity classes.)
One should stress that even with a powerful theory of symmetric functions from \cite{M} and \cite{L},
a passage from the monomial basis to the Schur basis is rather difficult: it is possible ``in theory''
but it is rather difficult in practice (of course, we speak here about ``large'' expressions).

It is, by no means, reasonable to ask why to work with Schur function expansions?
One of the aims of the present paper is (to try) to answer this question. 
Of course, an important role of Schur functions in geometry was known earlier, e.g., 
by the Schubert Calculus (see also \cite{L1}, \cite{P}, \cite{FP} -- to mention just a few references).
The latter reference gives a wide geometric motivation of the importance and ubiquity of
Schur functions in algebraic geometry. 

A basic property of Schur function expansions of Thom polynomials is the nonnegativity
of the coefficients proved by Andrzej Weber and the second named author in \cite{PW} (see also \cite{PW2}). 
These positive coefficients often have a pleasant algebraic
structure, e.g., satisfy some recursions. This allow one to organize the computations
of them in a pretty systematic way. Among these coefficients, we find numbers appearing
in different contexts in enumerative geometry, e.g., complete quadrics (see \cite{PI22}).
More, as it follows from a recent paper \cite{MPW},
the positivity of the coefficients of Schur function expansions of classical Thom
polynomials leads to upper bounds for the coefficients of Legendrian Thom polynomials
expanded in an appriopriate basis.

Another feature comes from the fact that Thom polynomials are closely related with degeneracy
loci of the cotangent map
$$
f^*: T^*N_M \to T^*M
$$ 
(by $T^*N_M$ we denote the cotangent bundle of $N$ pull backed by $f$ to $M$). 
Polynomials supported on
such degeneracy loci were described using Schur functions in \cite{P}; this
helps to study the Schur function expansions of Thom polynomials of other singularity classes.
 
In the present article, we survey basically only those papers, where the Schur function
expansions of Thom polynomials play a significant role in the process of their computations
or/and help in understanding their structure.

In \cite{FK}, the authors computed the Thom polynomial of the second order Thom-Boardman
singularity classes 
$$
\Sigma^{i,j}: M^m \to N^{m-i+1}
$$ 
via its Schur function expansion, and conjectured the positivity of Schur function expansion 
for all Thom-Boardman singularity classes. 

In \cite{PArx}, the second author stated some formulas for Thom polynomials 
of singularities $I_{2,2}, A_3: M^m \to N^{m+k}$ (any $k$) and some partial result for $A_i: M^m \to N^{m+k}$
(any $i,k$). These expressions had the form of Schur function expansions. 
The details were given in \cite{PI22}, \cite{PA3} and \cite{PAi}. In Sections \ref{Schurexp}
and \ref{recurrences}, we discuss some essential computations from these papers. 

In \cite{O1}, \cite{O2}, the first author computed Schur function expansions for $A_4: M^m \to N^{m+k}$ ($k=2,3$) 
and $III_{2,3}: M^m \to N^{m+k}$ (any $k$).

\medskip

This paper is organized as follows.

In Section \ref{Schur}, we recall the definition and properties of Schur functions, including: cancellation--,
vanishing--, basis-- , and factorization property.

In Section \ref{Thom}, we recall the notion of a singularity class, and, following Thom \cite{T}, attach to a singularity class its 
Thom polynomial.

In Section \ref{pideals}, we discuss the $\cP$-ideals of singularity classes. From the 
structure of the $\cP$-ideal of $\Sigma^i$, we deduce some result on a rectangle containment for partitions appearing
in the Schur function expansion of a Thom polynomial of $\Sigma \subset \Sigma^i$ (Theorem \ref{rect}).

In Section \ref{single_orbits}, We discuss a way of computing of Thom polynomials of the closures
of single R-L orbits in a space of jets of maps: $({\bf C}^{\bullet},0)\to({\bf C}^{\bullet+k},0)$, called there ``singularities''
after \cite{Riminvens}.
This is a ``method of restriction equations'' that we learned from \cite{Riminvens}.

In Section \ref{ch_eu}, we collect formulas for the Chern and Euler classes of singularities,
and show by an example, how one can compute them.

In Section \ref{Schurexp}, we state some general properties of the Schur function expansions of Thom
polynomials of singularities. Theorem \ref{rect} is reinterpreted for singularities. We discuss
the Thom polynomial of $III_{2,3}$ for any $k$. For any $i,k$, we give the $1$-part of the
Thom polynomial of $A_i$. We discuss also recent results of F\'eher and Rim\'anyi \cite{FeRim} giving a bound
on the lengths of partitions appearing in Schur function expansions, and certain basic recursion (on $k$).

In Section \ref{recurrences}, we recall Pascal staircases, and survey Schur function expansions of Thom polynomials
of $I_{2,2}$ and $A_3$ from \cite{PI22} and \cite{PA3}. Their coefficients obey some (other) recursions on $k$. We provide 
details of two computations with extensive use of the algebra of Schur functions and multi-Schur functions.

In Section \ref{III33}, we discuss the Schur function expansions of the Thom polynomials of $III_{3,3}$.

In Section \ref{I23}, we discuss some properties of the Thom polynomials of $I_{2,3}$.

In the appendices (Section \ref{appendix1} and \ref{appendix2}), we give the Schur function expansions of the Thom polynomials 
of $III_{3,3}$ and $I_{2,3}$ for several $k$.

\smallskip

This is basically a survey paper. Some new material is gathered in the last four sections. 
We lectured on this material at IMPANGA seminars in Warsaw and Cracow. 

\medskip

\noindent
{\bf Acknowledgments} \ We gratefully thank Alain Lascoux. He taught the second named author the Schur 
functions in 1979, and discussed with him the Schur function expansions of Thom polynomials in 2004. 
This was the starting point of the project surveyed in the present paper. He also taught, in 2008, the first 
named author how to write clever algorithms for computations with Schur functions.

We thank Maxim Kazarian for mailing us \cite{Ka}. We are also grateful to Alexander Klyachko and Andrzej Weber for helpful 
discussions. Finally, we thank the referee whose comments led to the improvement of the exposition.

A part of the present article was written during the stay of the second named author at RIMS in Kyoto, in March 2011.
He thanks this institute, and especially Shigeru Mukai, for the warm welcome there.

\section{Schur functions}\label{Schur}

The main reference for this section, for the conventions and notation, is \cite{L}. This book studies (among others) 
multi-Schur functions which are a useful generalization of Schur functions. We shall need them in this paper. But
we start our discussion with Schur functions.

For $m\in {\bf N}$, by an alphabet $\A$ of cardinality $m$ we shall mean a finite set of
indeterminates $\A=\{a_1,\dots,a_m\}$. Sometimes, to point out the cardinality of an
alphabet $\{a_1,\dots,a_m\}$ we shall denote it by $\A_m$.

We shall often identify an alphabet
$\{a_1,\dots,a_m\}$ with the sum $a_1+\cdots +a_m$.

\begin{definition}\label{cf}
Given two alphabets $\A$, $\B$, the {\it complete functions} $S_i(\AB)$
are defined by the generating series (with $z$ an extra variable):

\begin{equation}
\sum S_i(\AB) z^i:=\prod_{b\in \B} (1-bz)/\prod_{a\in \A}
(1-az)\,.
\end{equation}
\end{definition}

We see that $S_i(\A-\B)$ interpolates between $S_i(\A)$ - the complete symmetric function of degree $i$ in $\A$ 
and $S_i(-\B)$ - the elementary symmetric function of degree $i$ in $\B$ times $(-1)^i$. For example, $S_3(\A-\B)$ is equal to
$$
S_3(\A-\B)=S_3(\A)-S_2(\A)\Lambda_1(\B)+S_1(\A)\Lambda_2(\B)-\Lambda_3(\B)\,,
$$
where $\Lambda_i(\B)$ denotes the $i$-th elementary symmetric function in $\B$.

\medskip

A weakly increasing sequence $(i_1,i_2,\ldots,i_s)$ of nonnegative integers is called a {\it partition}. 
The number it divides into parts, $|I|=i_1+i_2+\cdots +i_s$, is called the {\it weight} of $I$. The nonzero
$i_p$ are called the {\it parts} of $I$. The number of nonzero parts is called the {\it length} of $I$. 

Given two partitions $I=(i_1,i_2,\ldots,i_s)$ and $J=(j_1,j_2,\ldots,j_t)$,
we shall say that $I$ {\it is contained in} $J$, and write $I\subset J$, if for any $p=0,1,2,\ldots$, we have $i_{s-p} \le j_{t-p}$. 

 Following \cite{L}, we give

\begin{definition}\label{sf}
Given a partition $I=(i_1, i_2, \ldots, i_s)\in
{\bf N}^s$, and alphabets $\A$ and $\B$, the {\it Schur function}
$S_I(\A - \B)$ is 
\begin{equation}\label{schur}
S_I(\A - \B):= \bigl|
     S_{i_q+q-p}(\A \moins \B) \bigr|_{1\leq p,q \le s}  \, .
\end{equation}
\end{definition}
In other words, we put on the diagonal from to bottom: $S_{i_1}, S_{i_2}, \ldots, S_{i_s}$, and then, in each column, 
the indices of the successive $S_j$'s should increase by one from bottom to top.
For example, if $I=(1,3,3,4,5)$, then
$$
S_I(\AB)
=\begin{vmatrix}
S_1(\AB)  &  S_4(\AB)  & S_5(\AB)  & S_7(\AB) &  S_9(\AB) \\
1 & S_3(\AB) & S_4(\AB) & S_6(\AB) & S_8(\AB) \\
0 & S_2(\AB) & S_3(\AB) & S_5(\AB) & S_7(\AB) \\
0 & S_1(\AB) & S_2(\AB) & S_4(\AB) & S_6(\AB) \\
0 & 1 & S_1(\AB) & S_3(\AB) & S_5(\AB)
\end{vmatrix}\,.
$$
%\end{document}

These functions are often called {\it supersymmetric Schur functions}
or {\it Schur functions in difference of alphabets}. See \cite{S}, \cite{BR}, \cite{P2},
\cite{PT}, \cite{M} and \cite{L} for their study.

\medskip

We have the following {\it cancellation property}: for alphabets $\A$, $\B$, $\C$,
\begin{equation}
S_I((\A + \C) - (\B + \C))=S_I(\A-\B)\,.
\end{equation}

We shall use the simplified notation $i_1i_2\cdots i_s$ or
$i_1,i_2,\ldots , i_s$ for a partition $(i_1,i_2,\ldots,i_s)$
(the latter one if $i_s\ge 10$). Also, we shall write $(i^s)$
for the partition $(i,\ldots,i)$ ($s$ times).

A partition $I$ has a graphical representation due to Ferrers, called
its {\it diagram}: it is a diagram of left packed square boxes with
$i_1,i_2,\ldots,i_s$ the number of boxes in the successive rows.
For example, the diagram of the partition $(2,5,6,8)$ is:

$$
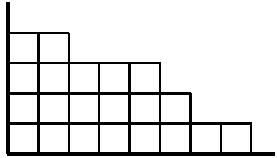
\unitlength=2mm
\begin{picture}(18,11)
%\put(4,10){\vector(1,0)4}
%\put(4,10){\vector(-1,0)4}
%\put(18,2){\vector(0,1)2}
%\put(18,2){\vector(0,-1)2}
%\put(4.2,10.3){\hbox to0pt{\hss$4$\hss}}
%\put(18.5,1.5){\hbox{$2$}}
%\put(8.7,4){*}
\put(0,0){\line(0,1){8}}
\put(0,0){\line(1,0){16}}
\put(0,2){\line(1,0){16}}
\put(0,4){\line(1,0){12}}
\put(0,6){\line(1,0){10}}
\put(0,8){\line(1,0){4}}
\put(2,0){\line(0,1){8}}
\put(4,0){\line(0,1){8}}
\put(6,0){\line(0,1){6}}
\put(8,0){\line(0,1){6}}
\put(10,0){\line(0,1){6}}
\put(12,0){\line(0,1){4}}
\put(14,0){\line(0,1){2}}
\put(16,0){\line(0,1){2}}
\linethickness{1.1pt}
\put(0,0){\line(0,1){10}}
\put(0,0){\line(1,0){17.5}}
%\put(8,4){\line(1,0){12}}
%\put(8,4){\line(0,1){9}}
\end{picture}
$$

Given two partitions $I$ and $J$, if we put their diagrams 
in such a position that they share the lowest row and the leftmost
column, then `` $I\subset J$ '' iff the set of boxes of the diagram of $I$ is contained
in the set of boxes of the diagram of $J$. 

We record the following property:
\begin{equation}
S_I(\A - \B)= (-1)^{|I|}S_J(\B - \A)=S_J(\B^* - \A^*)\,,
\end{equation}
where $J$ is the conjugate partition of $I$ (i.e. the consecutive
rows of the diagram of~$J$ are the transposed columns of the
diagram of $I$), and
$\A^*$ denotes the alphabet $\{-a_1,-a_2,\dots \}$.

Fix two positive integers $m$ and $n$. Let $I$ be a partition.
Suppose that the diagram of $I$ and the following $(m,n)$-{\it hook}~:

\smallskip

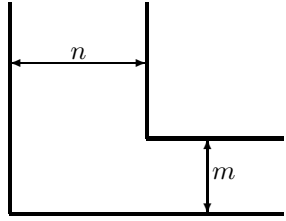
$$
\unitlength=2mm
\put(4,10){\vector(1,0)5}
\put(4,10){\vector(-1,0)4}
\put(13,2){\vector(0,1)3}
\put(13,2){\vector(0,-1)2}
\put(4.5,10.3){\hbox to0pt{\hss$n$\hss}}
\put(13.3,2.3){\hbox{$m$}}
\linethickness{1.1pt}
\begin{picture}(18,14)
\put(0,0){\line(0,1){14}}
\put(0,0){\line(1,0){18}}
\put(9,5){\line(1,0){9}}
\put(9,5){\line(0,1){9}}
\end{picture}
$$
share the lowest row and the leftmost column. If the diagram of $I$ is contained in this hook, then we say
that the partition $I$ {\it is contained in} the $(m,n)$-hook.

\medskip

We record the following {\it vanishing property}.
Given alphabets $\A$ and $\B$ of cardinalities $m$ and $n$, if
the diagram of a partition $I$ is not contained in the $(m,n)$-hook,
then
\begin{equation}\label{van}
S_I(\A-\B)=0\,.
\end{equation}

For instance, $I=(2,5,6,8)$ is not contained in the $(2,4)$-hook
$$
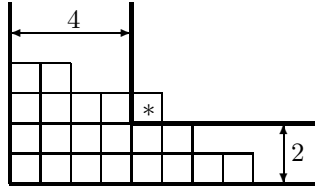
\unitlength=2mm
\begin{picture}(18,12)
\put(4,10){\vector(1,0)4}
\put(4,10){\vector(-1,0)4}
\put(18,2){\vector(0,1)2}
\put(18,2){\vector(0,-1)2}
\put(4.2,10.3){\hbox to0pt{\hss$4$\hss}}
\put(18.5,1.5){\hbox{$2$}}
\put(8.7,4){*}
\put(0,0){\line(0,1){8}}
\put(0,0){\line(1,0){16}}
\put(0,2){\line(1,0){16}}
\put(0,4){\line(1,0){10}}
\put(0,6){\line(1,0){10}}
\put(0,8){\line(1,0){4}}
\put(2,0){\line(0,1){8}}
\put(4,0){\line(0,1){8}}
\put(6,0){\line(0,1){6}}
\put(8,0){\line(0,1){6}}
\put(10,0){\line(0,1){6}}
\put(12,0){\line(0,1){4}}
\put(14,0){\line(0,1){2}}
\put(16,0){\line(0,1){2}}
\linethickness{1.1pt}
\put(0,0){\line(0,1){12}}
\put(0,0){\line(1,0){20}}
\put(8,4){\line(1,0){12}}
\put(8,4){\line(0,1){8}}
\end{picture}
$$
Therefore $S_{2568}(\A_2-\B_4)=0$. This vanishing property is an immediate consequence of
the factorization property (see Eq.(\ref{Fact})). 

\medskip

Moreover, we have the following result.
\begin{theorem}\label{Tss}
If $\A$ and $\B$ are alphabets of cardinalities $m$ and $n$, then 
the Schur polynomials $S_I(\A-\B)$, where $I$ runs over partitions
contained in the $(m,n)$-hook, are ${\bf Z}$-linearly independent.
(I.e., they form a basis of the abelian group of supersymmetric Schur functions
in $\A$ and $\B$.)
\end{theorem}
For a proof, see, e.g., \cite[Proposition 2.3]{PT}. 

\begin{note} \rm We shall often identify partitions with their diagrams, as is
customary. 
\end{note}

\smallskip

It is handy to adopt the following

\begin{convention} \rm
Instead of introducing in the argument of a symmetric function,
formal variables which will be specialized, we write
$\fbox{$r$}$ for a variable which will be specialized to $r$
($r$ can be $2x_1$, $x_1+x_2$,\ldots).
For example,
$$
S_2(x_1+x_2)=x_1^2+ x_1x_2+ x_2^2 \ \ \ \hbox{but} \ \ \
S_2\bigl(\fbox{$x_1\plus x_2$}\bigr)=
(x_1+x_2)^2= x_1^2+ 2x_1x_2+ x_2^2\,.
$$
\end{convention}
This convention stems from \cite{L2} where the reader can find instructive examples of its use. 

\begin{definition}
Given two alphabets $\A,\B$, we set
\begin{equation}\label{res}
R(\A,\B):=\prod_{a\in \A,\, b\in \B}(a - b)\,,
\end{equation}
the resultant of $\A, \B$.
\end{definition}
Thus $R(\A,\B)$ is the resultant of the polynomials $R(x,\A)=R(\{x\},\A)$ and $R(x,\B)$.

%We have (see \cite{L})
%\begin{equation}\label{ER}
%R(\A_m,\B_n)= S_{(n^m)}(\AB)=\sum_I S_I(\A) S_{(n^m)/I}(-\B)\,,
%\end{equation}
%where the sum is over all partitions $I\subset (n^m)$.

\bigskip

We now record some properties of Schur functions that are used in our computations with Thom polynomials.

The first one is the following {\it linearity formula}. We have (see \cite{L})
\begin{equation}\label{bq}
S_{j}(-\E-\B_{n})=S_{j}(-\E-\B_{n-1})-b_n S_{j-1}(-\E-\B_{n-1})\,.
\end{equation}

This equality is used quite often to estimate the sizes of partitions indexing Schur function expansions
of Thom polynomials (see, e.g., \cite{PI22}, \cite{PA3}, \cite{O1}, \cite{O2}, \cite{O3}). It serves also to establish 
an extremely useful {\it Transformation Lemma} (see Lemma \ref{TL}). 

\medskip

The second one is the following {\it factorization property} \cite{BR}.
For partitions $I=(i_1,\dots,i_m)$ and 
$J=(j_1,\dots, j_s)$, we have
\begin{equation}\label{Fact}
S_{(j_1,\dots,j_s,i_1+n,\dots,i_m+n)}(\A_m-\B_n)
=S_I(\A_m) \ R(\A_m,\B_n) \ S_J(-\B_n)\,.
\end{equation}

For example, with $m=4$, $n=2$, $I=(2,3)$, $J=(1,3)$, we have

        $$
\unitlength=2.5mm
\begin{picture}(40,12)
\put(0,0){\line(0,1){8}}
\put(0,0){\line(1,0){14}}
\put(0,2){\line(1,0){14}}
\put(0,6){\line(1,0){6}}
\put(0,8){\line(1,0){2}}
\put(2,0){\line(0,1){8}}
\put(4,0){\line(0,1){6}}
\put(6,0){\line(0,1){6}}
\put(10,0){\line(0,1){4}}
\put(12,0){\line(0,1){4}}
\put(14,0){\line(0,1){2}}
\put(0,4){\line(1,0){8}}
\put(18,5){$=$}
 \put(22,4.4){\line(1,0){6}}
\put(22,4.4){\line(0,1){4}}
\put(22,8.4){\line(1,0){2}}
\put(24,6.4){\line(0,1){2}}
\put(24,6.4){\line(1,0){4}}
\put(28,4.4){\line(0,1){2}}
\put(22,4){\line(1,0){8}}
\put(22,0){\line(1,0){8}}
\put(22,0){\line(0,1){4}}
\put(30,0){\line(0,1){4}}
\put(30.4,0){\line(0,1){4}}
\put(30.4,0){\line(1,0){6}}
\put(30.4,4){\line(1,0){4}}
\put(34.4,2){\line(0,1){2}}
\put(34.4,2){\line(1,0){2}}
\put(36.4,0){\line(0,1){2}}
\put(23,5){ $-\B_4$}
\put(25,1.5){$R$}
\put(32,1.5){ $\A_2$}
\linethickness{1.1pt}
\put(0,4){\line(1,0){8}}
\put(8,0){\line(0,1){6}}
\put(0,0){\line(0,1){10}}
\put(0,0){\line(1,0){16}}
\put(8,4){\line(1,0){8}}
\put(8,4){\line(0,1){6}}
\end{picture}
$$
$$S_{1367}(\A_2-\B_4)=S_{23}(\A_2)R(\A_2,\B_4)S_{13}(-\B_4)\, .$$
		
\smallskip

\noindent
This factorization property is useful to simplify the $h$-parts (cf. the end of Section \ref{pideals}) of Thom polynomials
(see Section \ref{recurrences}, and \cite{PI22}, \cite{PA3}, \cite{O1}, \cite{O2}). Cf. also \cite{algorithm}. 

\bigskip

We shall also need multi-Schur functions. Given $s$, two sets of alphabets $\{\A^1,\A^2,\ldots,\A^s\}$, $\{\B^1,\B^2,\ldots, \B^s\}$,
and partition $I=(i_1,\ldots,i_s)$, we define following \cite{L} the {\it multi-Schur function}
\begin{equation}
S_I(\A^1-\B^1,\ldots,\A^s-\B^s)=
\bigl|S_{i_q+q-p}(\A^q \moins \B^q) \bigr|_{1\leq p,q \le s}\,.
\end{equation}
In case where the alphabets are repeated, we indicate by a semicolon the corresponding bloc separation. For example, 
$$
S_{i,i;i}(\A-\C;\B-\D)=S_{i,i,i}(\A-\C,\A-\C,\B-\D)\,.
$$

\medskip

We record the following {\it Transformation Lemma} (see \cite[Lemma 1.4.1]{L})

\begin{lemma}\label{TL} Let $\D^0,\D^1,\ldots,\D^{s-1}$ be a family of alphabets such that $\card(\D^i)\le i$ for $0\le i\le s-1$. 
Then the multi-Schur function $S_I(\A^1-\B^1,\ldots,\A^s-\B^s)$ is equal to the determinant
$$
\bigl|S_{i_q+q-p}(\A^q \moins \B^q \moins \D^{s-p}) \bigr|_{1\leq p,q \le s}\,.
$$
\end{lemma}
In other words, one does not change the value of a multi-Schur function by replacing in row $p$ the difference $\A-\B$ by 
$\A-\B-\D^{s-p}$. We leave it to the reader to prove this result. 

\section{Thom polynomials of singularity classes of maps}\label{Thom}

Fix $m,n,p \in {\bf N}$. Consider the space ${\cJ}^p({\bf C}^m_0,{\bf C}^n_0)$ of $p$-jets of analytic
functions from ${\bf C}^m$ to ${\bf C}^n$ which map $0$ to $0$. Consider the natural right-left action of the group $\Aut_m^p \times \Aut_n^p$ on 
${\cJ}^p({\bf C}^m_0,{\bf C}^n_0)$, where $\Aut_n^p$ denotes the group of $p$-jets of automorphisms of $({\bf C}^n,0)$. By a {\it singularity class} we shall mean a closed algebraic right-left invariant subset of
${\cJ}^p({\bf C}^m_0,{\bf C}^n_0)$. Given complex analytic manifolds $M^m$ and $N^n$, a singularity class 
$\Sigma \subset {\cJ}^p({\bf C}^m_0,{\bf C}^n_0)$ defines the subset $\Sigma(M,N)\subset {\cJ}^p(M,N)$, where ${\cJ}^p(M,N)$ is the space of $p$-jets from $M$ to $N$.

\begin{theorem}\label{Thompol}
Let $\Sigma \subset {\cJ}^p({\bf C}^m_0,{\bf C}^n_0)$ be a singularity class.
There exists a universal polynomial ${\cT}^{\Sigma}$ over $\bf Z$
in $m+n$ variables $c_1,\ldots,c_m,c'_1,\ldots,c'_n$ which depends only on $\Sigma$, $m$ and $n$
such that for any complex analytic manifolds $M^m$, $N^n$ and for almost any map\footnote{The Riemann-Hurwitz formula quoted in Introduction
holds for any surjective $f$. In the theory of Thom polynomials we restrict ourselves only to {\it almost all} maps, i.e., the maps from
some open subset in the space of all maps.} $f: M\to N$,
the class of 
$$
\Sigma(f):=f_p^{-1}(\Sigma(M,N))
$$ 
is equal to
$$
{\cal T}^{\Sigma}(c_1(M),\ldots,c_m(M),f^*c_1(N),\ldots,f^*c_n(N)),
$$
where $f_p: M\to \cJ^p(M,N)$ is the $p$-jet extension of $f$.
\end{theorem}
This is a theorem due to Thom, see \cite{T}.

If a singularity class $\Sigma$ is {\it stable} (e.g. closed under
the contact equivalence, see, e.g., \cite{FeRim}), then ${\cal T}^{\Sigma}$ depends on $c_i(TM-TN_M)$.

\smallskip

Let $f: M\to N$ be a map of complex analytic manifolds. In the present paper, we shall work with the cotangent map
\begin{equation}\label{cot}
f^*: T^*N_M \to T^*M\,,
\end{equation}
rather than with the tangent one.
Given a partition $I$, we define
\[S_I(T^*M - T^*N_M)\]
to be the effect of the following specialization of $S_I(\A \moins \B)$: the indeterminates of $\A$
are set equal to the Chern roots of $T^*M$, and the indeterminates of $\B$
to the Chern roots of $T^*N_M$.

Given a singularity class $\Sigma$, the Poincar\'e dual of $\Sigma(f)$,  for almost any map $f:M \to N$, will be written in the form 
\begin{equation}\label{konwencja1}
\sum_I \alpha_I S_I(T^*M - T^*N_M)
\end{equation}
with integer coefficients $\alpha_I$. 

Accordingly, we shall write 
\begin{equation}\label{konwencja2}
\mathcal{T}^{\Sigma}=\sum_I\alpha_I S_I\,, 
\end{equation}
where $S_I$ is identified with $S_I(\A \moins \B)$ for the universal Chern roots $\A$ and $\B$.

\medskip

For example, consider the singularity class $\Sigma=\Sigma^i$. So, $m-i \le n$, and looking at the $(m-i)$th
degeneracy locus of the cotangent map (\ref{cot}),
we have 
$$
\cT^{\Sigma^{i}}=S_{(n-m+i)^i}\,,
$$
the Giambelli-Thom-Porteous formula (see \cite{Po}).

A basic result on Schur function expansions of Thom polynomials of singularity classes is

\begin{theorem} (\cite{PW}) \ Let $\Sigma$ be a nontrivial stable singularity class. Then
for any partition $I$, the coefficient $\alpha_I$ in the Schur function
expansion of the Thom polynomial
$$
{\cal T}^{\Sigma}= \sum \alpha_I S_I \,,
$$
is nonnegative and $\sum_I \alpha_I >0$.
\end{theorem}

This result was conjectured in \cite{FK}. Thus, it is not obvious. But its proof is almost obvious. 
The original proof in \cite{PW} used the classification space of singularities 
and the Fulton-Lazarsfeld theorem \cite{FL}. We now give an outline of another proof (of nonnegativity only),
communicated to the second author by Klyachko (Ankara, 2006) and, independently, 
by Kazarian \cite{Ka}.

\medskip

\noindent
{\bf Sketch of proof of the theorem}  \ First, using some Veronese map, we ``materialize'' all singularity classes
in sufficiently large Grassmannians.

We fix a singularity class $\Sigma$ and take the Schur function expansion of ${\cal T}^\Sigma$.
We take sufficiently large Grassmannian containing $\Sigma$ and such that specializing
${\cal T}^\Sigma$ in the Chern classes of the tautological (quotient) bundle $Q$, 
we do not lose any Schur summand. 

We identify by the Giambelli formula (see \cite{F}, p.146 and \cite{FP}, p.18, p.27), a Schur polynomial of $Q$
with the corresponding Schubert cycle.

To test a coefficient in the Schur function expansion of ${\cal T}^\Sigma$, 
we intersect $[\Sigma]$ with the corresponding {\it dual} Schubert cycle
(see \cite{F}, p.150). Using the Bertini-Kleiman theorem \cite{Kl0}, we put the cycles in 
a general position, so that we can reduce to set-theoretic intersection, 
which is nonnegative.
\qed

\begin{note} \rm If $\alpha_I\ne 0$, then we shall say that $I$ belongs to the indexing set 
of the Schur function expansion of $\cT^{\Sigma}$, or that the partition $I$ appears in the Schur function
expansion of $\cT^{\Sigma}$, or just $I$ appears in $\cT^{\Sigma}$. 
\end{note}

\medskip

It appears that this positivity result can be used to find upper bounds for the coefficients
of expansions of Legendrian Thom polynomials in a suitable basis, see \cite{MPW}. For the Lagrangian Thom polynomials, this is
the basis of the so called $\widetilde{Q}$-functions, see \cite{MPW1}.

\medskip

We record now a variant valid for {\it not necessary} stable singularity classes.

\begin{theorem} (\cite{PW2}) \ Let $\Sigma$ be a nontrivial singularity class. Then
for any partitions $I,J$, the coefficient $\alpha_{I,J}$ in the Schur function
expansion of the Thom polynomial
$$
{\cal T}^{\Sigma}= \sum \alpha_{I,J} S_I(T^*M) S_J(TN_M)      
$$
is nonnegative, and $\sum_{I,J} \alpha_{I,J} >0$.
\end{theorem}
(It is important that we use the cotangent bundle to the source $M$ and the tangent bundle to the target $N$.)
The latter result implies the former, see \cite{PW2}. This last paper contains also some variations 
on positivity of generalized Thom polynomials, and emphasizes the role of cone classes for
globally generated and ample vector bundles, following Fulton and Lazarsfeld.

\section{$\cP$-ideals of singularity classes}\label{pideals}

More generally, it is natural to consider the $\cP$-ideal of a singularity class $\Sigma$, denoted by ${\cP}^{\Sigma}$.
This is the subset in the polynomial ring ${\bf Z}[c_1,\ldots,c_m,c'_1\ldots,c'_n]$,
consisting of all polynomials $P$ which satisfy the following universality property. For any complex analytic manifolds $M^m$, $N^n$ 
and almost any map $f: M\to N$, 
$$
P(c_1(M),\ldots,c_m(M),f^*c_1(N),\ldots,f^*c_n(N))
$$ 
is supported on $\Sigma(f)$. (This means -- see \cite{P}, \cite{FP} -- that the class of a cycle on $M$ in $H(M,{\bf Z})$ 
is in the image of $H(\Sigma(f),{\bf Z})\to H(M,{\bf Z})$.)

\begin{note} \rm
These ideals were first studied (1988) in \cite{P} for the classes $\Sigma=\Sigma^i$. 
They were rediscovered (2004) in \cite{FeRim1} in the context of group actions.
\end{note}

For $\Sigma=\Sigma^i$, ${\cal P}^{\Sigma}$ is simply the ideal of polynomials
which -- after specialization to the Chern classes of $M$ and $N$ -- support
cycles in the locus $D$, where
$$
\dim\bigl(\Ker(f_*: TM \to TN_M)\bigr)\ge i \,.
$$
for almost any map $f: M \to N$. (This means that the class of a cycle on $M$ in $H(M,{\bf Z})$ 
is in the image of $H(D,{\bf Z})\to H(M,{\bf Z})$.)

Note that in terms of the cotangent map, $D$ is the locus where 
$$
\rank \bigl(f^*: T^*N_M \to T^*M\bigr) \le m-i\,,
$$
for almost any map $f: M \to N$.

Of course, the component of minimal degree of ${\cal P}^{\Sigma}$ is generated over ${\bf Z}$ by ${\cT^{\Sigma}}$.
Usefulness of $\cP$-ideals come from the following observation. Suppose that $\Sigma \subset \Sigma'$, where $\Sigma'$
is another singularity class.
Then ${\cT}^{\Sigma}$ belongs to ${\cP}^{\Sigma'}$. Thus if one knows the algebraic structure of ${\cP}^{\Sigma'}$,
one can use it to compute ${\cT}^{\Sigma}$. 
In this way, the degeneracy loci of the cotangent map (\ref{cot}) appear to be
useful objects to study Thom polynomials. 

Set ${\cP}^i:={\cP}^{\Sigma^i}$. By \cite{P} and \cite{P2}, one knows the algebraic structure of ${\cP}^i$,
i.e., a certain {\it finite} set of its algebraic generators (cf. \cite[Proposition 6.1]{P}), 
and its {\bf Z}-basis (cf. \cite[Proposition 6.2]{P}). The arguments combine geometry of
Grassmann bundles with algebra of Schur functions.

Before proceeding further, let us state the following result which is rather useful
to compute the Schur function expansions of Thom polynomials. Its setting is the
same as that of Theorem \ref{Thompol}.

\begin{theorem}\label{rect} (\cite{P}, \cite{PI22})
Suppose that a stable singularity class $\Sigma$ is contained in
$\Sigma^i$. Then all summands in the Schur function expansion of ${\cal T}^{\Sigma}$ 
are indexed by partitions containing $(n-m+i)^i$.
\end{theorem}
Thus the partitions not containing this rectangle cannot appear in the Schur function
expansion of ${\cT}^{\Sigma}$. 

This result seems to be quite obvious. However, its proof is not obvious. Let $\A$ and $\B$ 
be two alphabets such that 
$$
\sum c_i=\prod_{a\in \A} (1+a) \ \ \ \hbox{and} \ \ \  \sum c'_j=\prod_{b\in \B} (1+b).
$$
We have
\begin{proposition}\label{no}
No nonzero ${\bf Z}[c_1,\ldots,c_m]$-linear combination of the Schur functions $S_I(\A-\B)$'s, where all $I$'s do not 
contain $(n-m+i)^i$, belongs to ${\cP}^{i}$.
\end{proposition}
The idea of the proof is to interpret ${\cal P}^i$ as a ``generalized resultant'', and use some specialization trick.
For details, we refer the reader to the proof of ``Claim'' on p. 164 in \cite{P2}.

Thus, in particular, no nonzero ${\bf Z}$-linear combination of the $S_I(\A-\B)$'s, where all $I$'s do not 
contain $(n-m+i)^i$, belongs to ${\cal P}^i$.

\smallskip

Also, we have
\begin{proposition}\label{contains}
Any $S_I(\A-\B)$, where $I$ contains $(n-m+i)^i$ belongs to ${\cal P}^i$.
\end{proposition}
The idea of the proof is to use a desingularization of $D$ in the product of two Grassmann bundles, and apply appropriate
pushforward formulas. For details, see \cite[Proposition 3.2]{P}.

We are now ready to justify the theorem. Since $\Sigma$ is contained in $\Sigma^i$, the Thom polynomial $\cT^{\Sigma}$
belongs to $\cP^i$. By the stability assumption, the Thom polynomial ${\cal T}^{\Sigma}$ 
is a (unique) ${\bf Z}$-linear combination of the $S_I(\A-\B)$'s. Propositions \ref{contains} and \ref{no}
imply that only Schur functions indexed by partitions containing the rectangle $(n-m+i)^i$ appear in this sum.
\qed

\bigskip

In the computations of Thom polynomials, it is convenient to ``split'' them into pieces supported on
the consecutive degeneracy loci of the cotangent map (\ref{cot}). Let $\cT$ be the Thom polynomial of a singularity class.  
Following  \cite{PAi}, by the $h$-{\it part} of ${\cal T}$ we mean the sum of all Schur functions appearing 
in ${\cal T}$ (multiplied by their coefficients) such that the corresponding partitions satisfy 
the following condition: $I$ contains the rectangle partition $(n-m+h)^h$, but it does not contain the larger 
diagram $(n-m+h+1)^{h+1}$. The polynomial ${\cal T}$ is a sum of its $h$-parts, $h=1,2,\ldots$.

\section{Single R-L orbits}\label{single_orbits}

In the present paper, we shall mostly study Thom polynomials of singularities.

Let $k\geq 0$ be a fixed integer and $\bullet\in {\bf N}$. Two stable germs $\kappa_1, \kappa_2 :({\bf C}^{\bullet},0)\to ({\bf C}^{\bullet+k},0)$ are said to be right-left equivalent if there exist germs of biholomorphisms $\phi$ of $({\bf C}^\bullet,0)$ and  $\psi$ of $({\bf C}^{\bullet+k},0)$ such that $\psi \circ \kappa_1 \circ \phi^{-1} = \kappa_2.$  A suspension of a germ is its trivial unfolding: $(x,v)\mapsto (\kappa(x),v)$. Consider the equivalence relation (on stable germs $({\bf C}^{\bullet},0)\to ({\bf C}^{\bullet+k},0)$) generated by right-left equivalence and suspension. A \textit{singularity} $\eta$ is an equivalence class of this relation.\footnote{This terminology stems from \cite{Riminvens}; a singularity
corresponds to a single R-L orbit.}

According to Mather's classification (\cite{dPW} or \cite{AVGL}), singularities are in one-to-one
correspondence with finite dimensional (local) ${\bf C}$-algebras.
We shall use the following notation of Mather:

\medskip

-- $A_i$ \   will stand for the
stable germs with local algebra ${\bf C}[[x]]/(x^{i+1})$, $i\ge 0$;

\medskip

-- $I_{a,b}$ \ (of Thom-Boardman type $\Sigma^{2,0}$) for stable germs
with
local algebra ${\bf C}[[x,y]]/(xy, x^a+y^b)$, \ $b\ge a\ge 2$;

\medskip

-- $III_{a,b}$ \ (of Thom-Boardman type $\Sigma^{2,0}$) for stable germs
with local algebra ${\bf C}[[x,y]]/(xy, x^a, y^b)$, \ $b\ge a\ge 2$
(here $k\ge 1$).

\bigskip

With a singularity $\eta$, there is associated Thom polynomial $\cal T^{\eta}$ in the formal variables $c_1, c_2, \dots$ which
after the substitution of $c_i$ to
\begin{equation}
c_i(f^*TN-TM)=[c(f^*TN)/c(TM)]_i\,,
\end{equation}
for a general map $f:M\to N$ between complex analytic manifolds,
evaluates the Poincar\'e dual of $[\eta(f)]$, where $\eta(f)$
is the cycle carried by the closure of the set
\begin{equation}
\{x\in M : \hbox{the singularity of} \ f \ \hbox{at} \ x \
\hbox{is} \ \eta \}\,.
\end{equation}
%We shall express Thom polynomials $\cal T^{\eta}$ following the conventions (\ref{konwencja1}) and (\ref{konwencja2}).

%\smallskip

By $\codim(\eta)$, we mean the codimension of $\eta(f)$ in $X$. 

Codimensions of above singularities are as follows (cf. \cite[Chapter 8]{dPW}): 

\smallskip

-- $A_i$ \ associated with maps $({\bf C}^{\bullet},0)\to ({\bf C}^{\bullet+k},0)$), where 
 $i\ge 0$ and $k\ge 0$ has codimension $(k+1) i$.

\medskip

-- $I_{a,b}$ \ associated with maps $({\bf C}^{\bullet},0)\to ({\bf C}^{\bullet+k},0)$), where 
\ $b\ge a\ge 2$ and $k\ge 0$ has codimension $(k+1)(a+b-1)+1$.
\medskip

-- $III_{a,b}$  \ associated with maps $({\bf C}^{\bullet},0)\to ({\bf C}^{\bullet+k},0)$), where 
\ $b\ge a\ge 2$ and $k\ge 1$ has codimension $(k+1)(a+b-2)+2$.

\medskip

We shall now follow the approach in \cite{Riminvens}.
Let $\kappa: ({\bf C}^n,0) \to ({\bf C}^{n+k},0)$
be a prototype of  a singularity $\eta$.
It is possible to choose a maximal compact subgroup $G_\eta$ of the
\textit{right-left symmetry group}
\begin{equation}
\Aut \kappa = \{(\phi,\psi) \in {\Aut}_n \times {\Aut}_{n+k}: \psi \circ \kappa \circ \phi^{-1} = \kappa \}\,,
\end{equation}
such that images of its projections
to the factors ${\Aut}_n$ and ${\Aut}_{n+k}$ are
linear \footnote{By ${\Aut}_n$ we mean here the space of automorphisms of $({\bf C}^n,0)$.}. 
That is, projecting on the source ${\bf C}^n$
and the target ${\bf C}^{n+k}$, we obtain representations
$\lambda_1(\eta)$ and
$\lambda_2(\eta)$. Let $E_{\eta}'$ and $E_{\eta}$ denote the vector
bundles associated
with the universal principal $G_\eta$-bundles $EG_\eta \to BG_\eta$
that  correspond to $\lambda_1(\eta)$ and
$\lambda_2(\eta)$, respectively. The {\it total Chern
class},  $c(\eta)\in H^{*}(BG_\eta,{\bf Z})$, and the {\it Euler
class}, $e(\eta)\in H^{2\codim(\eta)}(BG_\eta,{\bf Z})$,  of $\eta$
are defined by
\begin{equation}
c(\eta):=\frac{c(E_{\eta})}{c(E_{\eta}')}\qquad \mbox{and} \qquad
e(\eta):=e(E_{\eta}')\, .
\end{equation}

We end this section by recalling the {\it method of restriction equations} due to Rim\'anyi et al. 

\begin{theorem} (\cite{Riminvens})
Let $\eta$ be a singularity. 
Suppose  that the number of singularities  of codimension less than or equal to $\codim(\eta)$ is finite. 
Moreover, assume that the Euler classes of all singularities of codimension smaller than $\codim(\eta)$
are not zero-divisors. Then we have
\begin{enumerate}
\item  if \ $\xi\ne \eta$ \ and \ $\codim(\xi)\le \codim(\eta)$, then
\ ${\mathcal T}^{\eta}(c(\xi))=0$;
\item  ${\mathcal T}^{\eta}(c(\eta))=e(\eta)$.
\end{enumerate}

\noindent
This system of equations (taken for all such $\xi$'s) determines the Thom polynomial ${\mathcal T}^{\eta}$ in a unique way. 
\end{theorem}

\medskip

Solving of these equations is rather difficult. This method is well suited for compter experiments, 
though the bounds of such computations are quite sharp.

\section{Computing the Chern and Euler classes}\label{ch_eu}

The Chern and Euler classes recalled in the present section were given in: \cite{Riminvens}, \cite{PI22}, \cite{O1}, \cite{O2} and \cite{O3}.  

Let $\eta: ({\bf C}^{\bullet},0) \to ({\bf C}^{\bullet+k},0)$ be a singularity in the sense of Section \ref{single_orbits}.

For $\eta=A_i$, a suitable maximal compact subgroup can be chosen as $G_{A_i}=U(1)\times U(k)$. The Chern class is
\begin{equation}
c(A_i)=\frac{1+(i+1)x}{1+x}\prod_{j=1}^k(1+y_j)\,,
\end{equation}
where $x$ and $y_1$,\ldots, $y_k$ are the Chern roots of the universal bundles on $BU(1)$ and $BU(k)$.
The Euler class is
\begin{equation}\label{eAi}
e(A_i)= i! \ x^i \ \prod_{j=1}^k (y_j-ix)\cdots (y_j-2x)(y_j-x)\,.
\end{equation}

In case of $\eta=I_{2,2}$, we consider  the extension
of $U(1)\times U(1)$ by ${\bf Z}/2{\bf Z}$. Denoting this group by $H$, a maximal compact subgroup is
$G_\eta= H \times U(k)$ for all $k\ge 0$. But to make
computations easier, we use the subgroup $U(1)\times U(1) \times U(k)$ as $G_\eta$ (cf. \cite{Riminvens}, p.502)). 
We have
\begin{equation}
c(I_{2,2})=\frac{(1+2x_1)(1+2x_2)}{(1+x_1)(1+x_2)}\prod_{j=1}^k(1+y_j)\,.
\end{equation}
Here $x_1, x_2$ and $y_1,\ldots, y_k$ are the Chern roots
of the universal bundles on two copies of $BU(1)$ and on $BU(k)$. The Euler class is
\begin{equation}
e(I_{2,2})=x_1x_2(2x_1-x_2)(2x_2-x_1)\prod_{j=1}^k (y_j-x_1)(y_j-x_2)
(y_j-x_1-x_2)\,.
\end{equation}

Next, we consider $\eta=III_{2,2}$. This time we use the maximal compact group $G_\eta=U(2)\times U(k\moins 1)$ for $k\ge 1$. 
We have
\begin{equation}
c(III_{2,2})=\frac{(1\plus 2x_1)(1\plus 2x_2)(1\plus x_1\plus x_2)}{(1\plus x_1)(1\plus x_2)}
\prod_{j=1}^{k-1}(1+y_j)\,,
\end{equation}
where $x_1, x_2$ and $y_1,\ldots,y_{k-1}$ denote the Chern roots of the
universal bundles on $BU(2)$ and $BU(k\moins 1)$. The Euler class is
\begin{equation}
e(III_{2,2})=(x_1x_2)^2(x_1\moins 2x_2)(x_2\moins 2x_1) \
\prod_{j=1}^{k-1}(x_1\moins y_j) \ \prod_{j=1}^{k-1}(x_2\moins y_j)\,.
\end{equation}

For the singularity $III_{2,3}$, we can use the action of the $U(1)\times U(1) \times U(k-1)$. We have
\begin{equation}
c(III_{2,3})=\frac{(1\plus 2x_1)(1\plus 3x_2)(1\plus x_1\plus x_2)}{(1\plus x_1)(1\plus x_2)}\prod_{j=1}^{k-1}(1+y_j)\,.
\end{equation}
This time $x_1, x_2$ and $y_1,\ldots, y_k$ are the Chern roots
of the universal bundles on two copies of $BU(1)$ and on $BU(k-1)$.
The Euler class is
\begin{equation}
\begin{split}
e(III_{2,3})= &
4x_1^2x_2^3(x_1-x_2)(x_1-3x_2)(x_2-2x_1)\\ \times & \prod_{j=1}^{k-1}(x_1-y_j)(x_2-y_j)(2x_2-y_j)\,.
\end{split}
\end{equation}

For the singularity $III_{3,3}$, the maximal compact group is $U(2)\times U(k-1)$. The Chern class is
\begin{equation}
c(III_{3,3})=\frac{(1+3x_1)(1+3x_2)(1+x_1+x_2)}{(1+x_1)(1+x_2)}\prod_{j=1}^{k-1}(1+y_j)\, ,
\end{equation}
where $x_1,x_2$ and $y_1,\ldots,y_{k-1}$ are the Chern roots
of the universal bundles $BU(2)$ and $BU(k-1)$. The Euler class is
\begin{equation}
\begin{split}
e(III_{3,3})  = &  4x_1^3x_2^3(3x_1-x_2)(3x_1-2x_2)(3x_2-x_1)(3x_2-2x_1) \\
 \times & \prod_{j=1}^{k-1}(x_1-y_j)(2x_1-y_j)(x_2-y_j)(2x_2-y_j)\,.
\end{split}
\end{equation}

We display now the Chern or/and Euler classes of some other singularities (we omit to interpret the variables $x_i$ and $y_j$).
We have

\begin{equation}
c(I_{a,b})=\frac{(1+\frac{a+b}{\gcd(a,b)}x_1)(1+\frac{ab}{\gcd(a,b)}x_2)}{(1+\frac{a}{\gcd(a,b)}x_1)(1+\frac{b}{\gcd(a,b)}x_2)}\prod_{j=1}^{k-1}(1+y_j)\,;
\end{equation}

\begin{equation}
e(I_{a,b})=\frac{a!b!a^{b-1}b^{a-1}x^{a+b}}{\gcd(a,b)^{a+b}}\prod_{j=1}^k\left(\prod_{i=1}^a(i \frac{\gcd(a,b)}{b} x-y_j)\prod_{i=1}^{b-1}(i \frac{\gcd(a,b)}{a} x -y_j)\right) \,;
\end{equation}

\begin{equation}
c(III_{a,b})=\frac{(1+ax_1)(1+bx_2)(1+x_1+x_2)}{(1+x_1)(1+x_2)}\prod_{j=1}^{k-1}(1+y_j)\,;
\end{equation}

\begin{equation}
\begin{split}
e(III_{a,b}) &= (a-1)!(b-1)! \prod_{i=1}^{b-1}(ax_1-ix_2) \prod_{i=1}^{a-1}(bx_2-ix_1)\\
&\times \prod_{j=1}^{k-1}\left(\prod_{i=1}^{a-1}(y_j-ix_1)\prod_{i=1}^{b-1}(y_j-ix_2)\right)\,.
\end{split}
\end{equation}

A general strategy for computing the Chern and Euler classes of singularities was described in \cite{Riminvens}.

We show now, following \cite{O2}, how to compute the Euler class of $III_{2,3}$.
Assume that $k=1$ and consider the germ $g(x,y)=(x^2, y^3, xy)$. A prototype of $III_{2,3}$ can be written as the unfolding 
$$
g + \sum_{i=1}^8 u_i h_i\,,
$$ 
where $h_i$ form a basis of the space

$$\frac{\mathfrak{m}^3_{x,y}}{\mathfrak{m}_{x,y}\cdot \{\frac{\partial g}{\partial x}, \frac{\partial g}{\partial y}\} + {\bf C}^3 \cdot I(g)}\, ,$$ and where $I(g)$ is the subspace generated by the component functions of $g$. We shall work with the basis consisting of the following germs:
\begin{eqnarray*}
h_1(x,y)=(x,0,0), &  &  h_5(x,y)=(0,y,0), \\
h_2(x,y)=(y,0,0), &  &  h_6(x,y)=(0,y^2,0),\\
h_3(x,y)=(y^2,0,0),  &  & h_7(x,y)=(0,0,x),\\
h_4(x,y)=(0,x,0),  &  & h_8(x,y)=(0,0,y).\\
\end{eqnarray*}
Let $\rho_{h_i}$ denote the representation of the action of the group $U(1)\times U(1)$ on the space generated by $h_i$. Then, denoting the  one-dimensional representations of the first and the second copies of $U(1)$ by $\lambda$ and $\mu$, we have
\begin{eqnarray*}
\rho_{h_1}=\lambda, &  &  \rho_{h_5}=\mu^2, \\
\rho_{h_2}=\lambda^2 \otimes \mu^{-1}, &  &  \rho_{h_6}=\mu,\\
\rho_{h_3}=\lambda^2 \otimes \mu^{-2},  &  & \rho_{h_7}=\mu,\\
\rho_{h_4}=\lambda^{-1}\otimes \mu^3,  &  & \rho_{h_8}=\lambda.\\
\end{eqnarray*}
Therefore for $k=1$, using the representation $\bigoplus \rho_{h_i}$, we can write the Euler class as 
\begin{equation}
e(III_{2,3})=
4x_1^2x_2^3(x_1-x_2)(x_1-3x_2)(x_2-2x_1)\,,
\end{equation}
where $x_1$ and $x_2$ denote the Chern roots of the universal bundles on the two copies of $BU(1)$. 

For $k=2$, in addition to $h_i$ above, we need to consider the representations of the action of the group $U(k-1)=U(1)$ on the spaces generated by $(x,y)\mapsto(0,0,0,x)$, $(x,y)\mapsto(0,0,0,y)$ and $(x,y)\mapsto(0,0,0,y^2)$. These can be written as $\nu\otimes \lambda^{-1}$, $\nu\otimes \mu^{-1}$ and $\nu\otimes \mu^{-2}$, where $\nu$ denotes the  one-dimensional representation of this copy of $U(1)$. Hence, in this case, the Euler class can be written as   
\begin{equation}
e(III_{2,3})=
4x_1^2x_2^3(x_1-x_2)(x_1-3x_2)(x_2-2x_1)(x_1-y_1)(x_2-y_1)(2x_2-y_1)\,,
\end{equation}
where $x_i$ are as above and $y_1$ denotes the Chern root of the universal bundle on $BU(1)$.

For $k\geq 1$, we need to consider $U(k-1)$ instead of $U(1)$, giving rise to $y_1,\ldots,y_{k-1}$ (and respectively to the product
$\prod_{j=1}^{k-1}(x_1-y_j)(x_2-y_j)(2x_2-y_j)$) instead of $y_1$ (and respectively of $(x_1-y_1)(x_2-y_1)(2x_2-y_1)$).

\bigskip

We shall need the following alphabets:

\begin{definition} We set
\begin{align*}
\D &= \fbox{$x_1$}+\fbox{$x_2$}+\fbox{$x_1+x_2$}\,,\\
\E &= \fbox{$2x_1$}+\fbox{$2x_2$}\,,\\
\F &=  \fbox{$2x_1$}+\fbox{$3x_2$}+\fbox{$x_1+x_2$}\,,\\
\G &=  \fbox{$3x_1$}+\fbox{$3x_2$}+\fbox{$x_1+x_2$}\,,\\
\H &=  \fbox{$2x_1$}+\fbox{$4x_2$}+\fbox{$x_1+x_2$}\,.
\end{align*}
\end{definition}

\bigskip

\begin{notation} \rm
In the rest of the paper we shall use the shifted parameter
\begin{equation}
r:=k+1\,.
\end{equation}
When we need to emphasize the dependence on $r$ we shall write $\eta(r)$ for the
singularity $\eta :({\bf C}^{\bullet},0)\to ({\bf C}^{\bullet+r-1},0)$, and denote the 
Thom polynomial of $\eta(r)$ by $\mathcal{T}_r^{\eta}$, or $\cT_r$ for short.
(In this notation, the result of Thom, ${\cal T}_r^{A_1}=S_r$, has a transparent form.)
\end{notation}

\bigskip

We now specify, with the help of these alphabets, some equations characterizing Thom polynomials $\cT_r$ imposed
by different singularities. 

\begin{note} \rm The variables below will be specialized to the Chern roots of the {\it cotangent} bundles.
\end{note} 

First, we give the vanishing equations
coming from the Chern classes of singularities. Let $\B_j$ denote an alphabet
of cardinality $j$. We have the following equations:

\begin{equation}\label{air}
A_i(r): \ \ \ \ \ {\cT}_r(x\moins \B_{r-1}-\fbox{$(i+1)x$})=0 \ \ \ \ \ \ \ \hbox{for} \ \ i=0,1,2,\ldots \,;
\end{equation}

\begin{equation}\label{I22r}
I_{2,2}(r): \ \ \ \ \ {\mathcal T}_r \left(\X_2-\E-\B_{r-1}\right) = 0\,;
\end{equation}

\begin{equation}\label{I23r}
I_{2,3}(r): \ \ \ \ \ {\mathcal T}_r \left(\fbox{$2x$}+\fbox{$3x$}-\fbox{$5x$}-\fbox{$6x$}-\B_{r-1}\right) =0\,;
\end{equation}

\begin{equation}\label{III22r}
III_{2,2}(r): \ \ \ \ \ {\mathcal T}_r(\X_2-\D-\B_{r-2})=0\,;
\end{equation}

\begin{equation}\label{III23r}
III_{2,3}(r): \ \ \ \ \ {\mathcal T}_r \left(\X_2-\F-\B_{r-2}\right) = 0\,;
\end{equation}

\begin{equation}\label{III24r}
III_{2,4}(r): \ \ \ \ \ {\mathcal T}_r \left(\X_2-\H-\B_{r-2}\right) = 0\,.
\end{equation}
Using the Chern classes displayed above, one can write down other vanishing equations.

\bigskip

We give now some normalizing equations coming from the Euler classes of singularities. We have

\begin{equation}\label{nair}
A_i(r): \ \ \ \ \ {\cT}_r(x\moins \B_{r-1}\moins \fbox{$(i\plus 1)x$})=
R(x\plus \fbox{$2x$}\plus \fbox{$3x$}\plus
\cdots\plus\fbox{$ix$}\,, \B_{r-1}\plus \fbox{$(i\plus 1)x$}\, )\,;
\end{equation}

\begin{equation}\label{nI22r}
I_{2,2}(r): \ \ \ \ \ {\cT}_r(\X_2\moins \E \moins \B_{r-1})
=x_1x_2(x_1\moins 2x_2)(x_2\moins 2x_1)
\ R(\X_2\plus \fbox{$x_1\plus x_2$},\B_{r-1})\,;
\end{equation}

\begin{equation}\label{nI23r}
\begin{split}
I_{2,3}(r): \ \ \ \ \ 
&{\cT}_r \left(\fbox{$2x$}+\fbox{$3x$}-\fbox{$5x$}-\fbox{$6x$}-\B_{r-1}\right) \\
&= 2 x R(\fbox{$2x$}+\fbox{$3x$}\, ,\fbox{$5x$}+\fbox{$6x$}+\B_{r-1}) \\
& \times \prod_{j=1}^{r-1}(4x-b_j)(6x-b_j)\,;
\end{split}
\end{equation}

\begin{equation}\label{nIII22r}
III_{2,2}(r): \ \ \ \ \ {\cT}_r(\X_2-\D - \B_{r-2})
=R(\X_2, \D + \B_{r-2})\,;
\end{equation}

\begin{equation}\label{nIII23r}
III_{2,3}(r):
{\cT}_r \left(\X_2-\F-\B_{r-2}\right) = 2  x_2  (x_1-x_2) 
R(\X_2,\F+\B_{r-2})  \prod_{j=1}^{r-2}(2x_2 - b_j) \,;
\end{equation}

\begin{equation}\label{nIII33r}
\begin{split}
III_{3,3}(r):
 {\mathcal T}_r \left(\X_2-\G-\B_{r-2}\right)= & x_1 x_2 (3x_1-2x_2)(3x_2-2x_1)\\
  \times & R(\X_2,\G+\B_{r-2}) \prod_{j=1}^{r-2}(2x_1-b_j)(2x_2-b_j)\,.
\end{split}
\end{equation}

Using the Euler classes displayed above, one can write down other normalizing equations.

\section{Thom polynomials of singularities}\label{Schurexp}

In this section, we shall study, for singularities $\eta$, Schur function expansions Thom polynomials ${\cal T}^{\eta}$ written 
in the form (\ref{konwencja2}) (cf. also (\ref{konwencja1}):
$$
\mathcal{T}^{\eta}=\sum_I\alpha_I S_I\,.
$$ 
It is interesting to find bounds on partitions appearing in Schur function expansions of Thom
polynomials of singularities. One such follows immediately from Theorem \ref{rect}.

\begin{proposition}
Suppose that a singularity $\eta$ is of Thom-Boardman type $\Sigma^{i,\ldots}$. Then all summands
in the Schur function expansion of ${\cal T}^{\eta}_r$ are indexed
by partitions containing the rectangle partition $(r+i-1)^i$.
\end{proposition}

\bigskip

For example, consider the singularity $III_{2,3}(r)$.
As its Thom-Boardman type is $\Sigma^{2,0}$, all partitions in the Schur function expansion of 
$\cT^{III_{2,3}}(r)$ contain the partition $(r+1,r+1)$. This Thom polynomial is characterized by the equations: 
(\ref{air}), $i=0,1,2,3$, (\ref{III22r}) and (\ref{nIII23r}). Its Schur function expansion is given by the following expression: 

\begin{theorem} (\cite{O2}, \cite{FeRim}) We have 
\begin{equation}
{\mathcal T}_r^{III_{2,3}}=\sum_{i=1}^{r+1} 2^{i} S_{r+1-i , r+1 , r+i}\, .
\end{equation}
\end{theorem}

\subsection{On Morin singularities $A_i(r)$ }

One of the most important problems in global singularity theory is to write down the explicit Schur function expansion 
of the Thom polynomials for Morin singularities $A_i(r)$. We now describe, following \cite{PAi}, the $1$-part of ${\cal T}^{A_i}_r$
for any $i$ and $r$.

Let $\A$ be an alphabet of cardinality $m$. Consider the function $F(\A, -)$, defined for any difference of alphabets $\G-\H$ by
\begin{equation}
F(\A, \G-\H):= \sum_{I} S_I(\A) S_{n-i_m,\ldots,n-i_1,n+|I|}(\G-\H)\,,
\end{equation}
where the sum is over partitions $I=(i_1, i_2, \ldots, i_m)$ such that $i_m\leq n$. 

A basic link of this function to resultants is given by the following result.
\begin{lemma}\label{LFR} For a variable $x$ and an alphabet $\B$ of cardinality $n$, we have
\begin{equation}
F(\A,x-\B)= R(x+\A x,\B)\,.
\end{equation}
\end{lemma}
({\it loc.cit.} Lemma 8).

Next, we define the following function $F^{(i)}_r(-)$:

\begin{equation}
F^{(i)}_r(\G-\H)=\sum_J  S_J(\fbox{$2$}+\fbox{$3$}
+\cdots+\fbox{$i$}) S_{r-j_{i-1},\ldots,r-j_1,r+|J|}(\G-\H)\,,
\end{equation}
where the sum is over partitions $J\subset (r^{i-1})$,
and for $i=1$ we understand $F^{(1)}_r(-)=S_r(-)$.

The following result gives the key algebraic property of $F^{(i)}_r$.
\begin{proposition}\label{FBr} \ We have
\begin{equation}\label{Br}
F^{(i)}_r(x-\B_r)= R(x+\fbox{$2x$}+\fbox{$3x$}+\cdots
+\fbox{$ix$}\,, \B_r)\,.
\end{equation}
\end{proposition}
\proof
The assertion follows from Lemma \ref{LFR} with $m=i-1$, $n=r$, and
$\A=\fbox{$2$}+\fbox{$3$}+\cdots +\fbox{$i$}.$
\smallskip

With the help of Proposition \ref{FBr}, the following result on Thom polynomials
was established:
\begin{theorem}
For any $i,r$, the $1$-part of ${\cal T}^{A_i}_r$ is equal to $F^{(i)}_r$.
\end{theorem}
({\it loc.cit.} pp.173--174).

%\end{document}

We shall now use a couple of functions $F^{(i)}_r$ to rephrase some results from \cite{T}, \cite{Ro}, folklore, \cite{G} and \cite{Riminvens},
respectively:
$$
\aligned
F^{(1)}_r&=S_r={\cal T}^{A_1}_r ;\\
F^{(2)}_r&=\sum_{j\le r} 2^j S_{r-j,r+j}={\cal T}^{A_2}_r ;\\
F^{(3)}_1&=S_{111}+5S_{12}+6S_3={\cal T}^{A_3}_1 ;\\
F^{(4)}_1&+10S_{22}=S_{1111}+9S_{112}+26S_{13}+24S_4+10S_{22}={\cal T}^{A_4}_1 ;\\
F^{(3)}_2&+5S_{33}=S_{222}+5S_{123}+6S_{114}+19S_{24}+30S_{15}+36S_6+5S_{33}={\cal T}^{A_3}_2
\endaligned
$$
(\cite{PAi}, pp.174--176). The reader can find in \cite{PAi} more examples of the functions $F^{(i)}_r$. In the next section,
we shall discuss the Schur function expansions of ${\cal T}^{A_3}_r$ for all $r$.

\smallskip

\begin{definition}\label{phi} For a positive integer $p$ We denote by $\Phi_p$ the linear endomorphism on the ${\bf Z}$-module
spanned by Schur functions indexed by partitions of length $\le p$
that sends a Schur function $S_{j_1,\ldots, j_p}$ to $S_{j_1+1,\ldots,j_p+1}$.
\end{definition}

\begin{example} \rm For any $i,r \ge 1$, we have
\begin{equation}\label{FF}
F^{(i)}_r=\overline{F^{(i)}_r}+{\Phi_i}(F^{(i)}_{r-1})\,,
\end{equation}
where the first summand gathers the Schur functions indexed by partitions of length $<i$. 
\end{example}

\smallskip

In \cite{B}, the author discusses another approach to Thom polynomials of Morin singularities.

\subsection{A basic recursion}

In the forthcoming section, we shall discuss some recursions for Thom polynomials. The following result was recently
obtained in \cite[Proposition 7.15, Theorem 7.14]{FeRim}. Let $Q_{\eta}$ denote the local algebra of the singularity $\eta$.

\begin{theorem} Let $\eta$ be a stable singularity. Then the length of any partition, appearing in the Schur function expansion
of $\cT^{\eta}_r$, is $\le \dim(Q_{\eta})\moins 1$. Moreover, by erasing one column of length $\dim(Q_{\eta})\moins 1$ from all the diagrams
of partitions appearing in $\cT^{\eta}_r$, we get all the diagrams of partitions appearing in ${\cal T}^{\eta}_{r-1}$ (we disregard the partitions whose diagrams have no such a column).
\end{theorem}

In other words, for $p=\dim(Q_{\eta})-1$, the following equation holds:
\begin{equation}\label{phit}
\cT^{\eta}_r=\overline{\cT}^{\eta}_r + {\Phi_p}(\cT^{\eta}_{r-1})\,,
\end{equation}
where the first summand gathers the Schur functions indexed by partitions of length $<p$.

This result was earlier established for the singularities $I_{2,2}(r)$, $A_3(r)$, $A_4(r)$, $III_{2,3}(r)$ and $III_{3,3}(r)$ from the restriction
equations which they obey, with help of Eq.(\ref{bq}) (see \cite{PI22}, \cite{PA3}, \cite{O1}, \cite{O2} and \cite{O3}).

This recurrence relation is quite easy to observe, especially by computing examples with the help of computer.
It is, however, not sufficient to compute Thom polynomials. As the matter of fact, Schur function expansions of Thom polynomials
often contain many terms, where the first column is shorter than the maximal possible. So these ``initial terms'', denoted
by $\overline{\cT}^{\eta}_r$ in (\ref{phit}), cannot be obtained by the operation of adding a maximal possible column. 

Another interesting question is to find upper bounds of the coefficients in Schur function expansions of Thom polynomials.
This will be a subject of some future study.

\section{Pascal staircases and two recursions}\label{recurrences}

We invoke first some results from \cite{PI22} and \cite{PA3}. We start with useful algebraic identity
associated with Pascal staircases (cf. \cite{PA3}). Then we discuss the Schur function computations of the Thom polynomials
of $I_{2,2}(r)$ and $A_3(r)$.

\subsection{Pascal staircases}\label{pascal}

The material of this subsection stems from \cite{PA3}.

Consider an infinite matrix $P=[p_{s,t}]$ with rows and columns numbered
by $s,t=1,2, \ldots$. 

We suppose that 
$p_{1,t}=p_{2,t}=0$ for $t\ge 2$, $p_{3,t}=p_{4,t}=0$
for $t\ge 3$, $p_{5,t}=p_{6,t}=0$ for $t\ge 4$ etc.

The first column is an arbitrary sequence $v=(v_1,v_2,\ldots)$. In the case when 
this sequence is the sequence of coefficients of the Taylor expansion of a function $f(z)$, 
we write $P_f$ for the corresponding matrix $P$.

To define the remaining $p_{s,t}$'s, we use the recursive formula
\begin{equation}
p_{s+1,t}= p_{s,t-1} + p_{s,t}.
\end{equation}
We visualize this definition by
$$ \begin{matrix}   a & b\\ & \square \end{matrix}   \qquad
\Rightarrow   \qquad \begin{matrix}   a & b\\ & a+b \end{matrix}
$$
We thus get the following {\it Pascal staircase} $P=[p_{i,j}]_{i,j=1,2,\ldots}$:
$$
\begin{array}{cccccc}
v_1 & 0 & 0 & 0 & 0 & \ldots \\
v_2 & 0 & 0 & 0 & 0 & \ldots \\
v_3 & v_2 & 0 & 0 & 0 & \ldots \\
v_4 & v_3\plus v_2 & 0 & 0 & 0 & \ldots \\
v_5 & v_4\plus v_3\plus v_2 & v_3\plus v_2 & 0 & 0 & \ldots \\
v_6 & v_5\plus v_4\plus v_3\plus v_2 & v_4\plus 2v_3\plus 2v_2 & 0 & 0 & \ldots \\
v_7 &  v_6\plus v_5\plus v_4\plus v_3\plus v_2 & v_5\plus 2v_4\plus 3v_3\plus 3v_2 & v_4\plus 2v_3\plus 2v_2& 0 & \ldots \\
\vdots & \vdots & \vdots & \vdots & \vdots &
\end{array}
$$

Given an integer $d \ge 0$, and an alphabet $\A$, we define the function
$W(d)=W(d,\A)$ by
\begin{equation}
W(d,\A)=\sum_{i,j} p_ {d+1-i,j+1}\, S_i(-\A)\, S_{j,d-i-j}(\X_2).
\end{equation}

The function $W(d,\A)$ is linear in the elements of the first 
column of $P$. Hence it is sufficient to restrict to the case 
$v=(1,y,y^2,\ldots)$, i.e., to take \ $P= P_{1/(1-zy)}$ to determine it.

\begin{lemma} If $P= P_{1/(1-zy)}$ and $\A= \fbox{$x_1+x_2$}$, then $W(0)=1$ 
and for $d\ge 1$
\begin{equation}\label{lw}
W(d,\fbox{$x_1+x_2$}) = (y-1)y^{d-1} S_d(\X_2).
\end{equation}
\end{lemma}
For the proof, see \cite{PA3}.

Let $\B$ be another alphabet. Taking now $\A=\fbox{$x_1+x_2$}+\B$ instead of 
$\fbox{$x_1+x_2$}$, and using  
\begin{multline*} 
W(d,\A)=
\sum_{i,j,k} p_ {d+1-i-k,j+1}\, S_i\bigl(-\fbox{$x_1+x_2$}\bigr)\, 
S_{j,d-i-j-k}(\X_2) S_k(-\B) \\
= \sum_k W\left(d-k,\fbox{$x_1+x_2$}\right) S_k(-\B)  \\
= (1-y^{-1}) \sum_k y^{d-k} S_{d-k}(\X_2) S_k(-\B) =  y^t(1-y^{-1}) S_d( \X_2-y^{-1}\B)\,,  
\end{multline*}
we get the following corollary.

\begin{corollary}\label{pcorollary} If $P= P_{1/(1-zy)}$ and $\B$ is an arbitrary alphabet,
then $($apart from initial values$)$ we have
\begin{equation}
W(d,\fbox{$x_1+x_2$}+\B)=(y-1)y^{d-1} S_d(\X_2-y^{-1}\B)\,.
\end{equation}
\end{corollary}

\subsection{Recursions for $I_{2,2}(r)$}

The material of this subsection stems from \cite{PI22}.

The codimension of $I_{2,2}(r)$, $r\ge 1$, is $3r+1$. Set $\cT_r:=\cT_r^{I_{2,2}}$ and $\overline{\cT}_r=\overline{\cT}_r^{I_{2,2}}$. We have $\cT_1=\overline{\cT}_1=S_{22}$.

A partition appearing in the Schur function expansion
of ${\cal T}_r$ contains the partition $(r+1,r+1)$ and has at most three parts.
In particular, if $S_{i_1,i_2}$ appears in the Schur function expansion
of ${\cal T}_r$, then $i_1=r+1+p$ and $i_2=2r-p$, where $0\le 2p\le r-1$.

Invoke the map $\Phi_3$ from Definition \ref{phi}. We have, for $r\ge 2$, the following recursive equation:

\begin{equation}\label{cr}
{\cal T}_r={\overline{\cal T}_r}+\Phi_3({\overline{\cal T}_{r-1}})
+\Phi_3^2({\overline{\cal T}_{r-2}})+\cdots+\Phi_3^{r-1}
({\overline{\cal T}_1)}\,.
\end{equation}

So we are left with computation of $\overline{\cal T}_r$.

Consider the matrix whose $(i,j)$th entry is the partition $(i+j,1+2i-j)$ with the convention that $(i+j,1+2i-j)$ 
is the empty partition for $2j>i+1$: 
$$
\left[\begin{array}{cccccc}
22 & \emptyset & \emptyset & \emptyset & \emptyset & \ldots \\
34 & \emptyset & \emptyset & \emptyset & \emptyset & \ldots \\
46 & 55 & \emptyset & \emptyset & \emptyset & \ldots \\
58 & 67 & \emptyset & \emptyset & \emptyset & \ldots \\
6,10 & 79 & 88 & \emptyset & \emptyset & \ldots \\
7,12 & 8,11 & 9,10 & \emptyset & \emptyset & \ldots \\
8,14 & 9,13 & 10,12 & 11,11 & \emptyset & \ldots \\
\vdots & \vdots & \vdots & \vdots & \vdots &
\end{array}\right]
$$

Note that the $r$th row of the above matrix consists of partitions appearing in ${\overline{\cal T}_r}$. It turns out that the coefficients
of their Schur functions are given by the corresponding entries of the Pascal staircase $P=[P_{i,j}]_{i=1,\ldots;j=1,\ldots}$, associated with the sequence $\{2^i-1\}_{i=1,2,\ldots}$:

\begin{equation}\label{pstair1}
P=\left[\begin{matrix}
1    & 0    & 0    & 0     & 0    & 0   & \ldots &   \\
3    & 0    & 0    & 0     & 0    & 0   & \ldots &   \\
7    & 3    & 0    & 0     & 0    & 0   & \ldots &   \\
15   & 10   & 0    & 0     & 0    & 0   & \ldots &   \\  
31   & 25   & 10   & 0     & 0    & 0   & \ldots &   \\
63   & 56   & 35   & 0     & 0    & 0   & \ldots &   \\
127  & 119  & 91   & 35    & 0    & 0   & \ldots &   \\
\vdots  & \vdots  &  \vdots & \vdots   &  \vdots & \vdots &  &  \\
\end{matrix}\right]
\end{equation}

Namely, we have
$$
{\overline{\cal T}_r}=\sum_{2j\leq r+1} P_{r,j} S_{r+j,2r+1-j}\,.
$$

\begin{example}
We have the following values of ${\overline{\cal T}_1}, \ldots, {\overline {\cal T}_7}$:
$$
\aligned
{\overline{\cal T}_1} &= S_{22}\\
{\overline{\cal T}_2} &= 3S_{34}\\
{\overline{\cal T}_3} &= 7S_{46}+3S_{55}\\
{\overline{\cal T}_4} &= 15S_{58}+10S_{67}\\
{\overline{\cal T}_5} &= 31S_{6,10}+25S_{79}+10S_{88}\\
{\overline{\cal T}_6} &= 63S_{7,12}+56S_{8,11}+35S_{9,10}\\
{\overline{\cal T}_7} &= 127S_{8,14}+119S_{9,13}+91S_{10,12}+35S_{11,11}\,.
\endaligned
$$
\end{example}

In this case, the algebra of Schur functions combined with one of the equations
characterizing the Thom polynomial, yields quickly an expression for ${\overline{\cal T}_r}$.
Of course, ${\overline{\cal T}_r}$ is uniquely determined by its value on
$\X_2$. The following result gives this value.

\begin{proposition}\label{Po}
For any $r\ge 1$, we have
\begin{equation}
{\overline{\cal T}_r}(\X_2)=(x_1x_2)^{r+1} \ S_{r-1}(\D)\,.
\end{equation}
\end{proposition}
We show the induction step.
Suppose that the assertion is true for ${\overline{\cal T}_i}$,
where $i<r$.
Let $I=(j,r+1+p,r+1+q)$ be a partition 
appearing in the Schur function expansion of ${\cal T}_r$.
By the factorization property (\ref{Fact}), we get
$$
S_I(\X_2-\D-\B_{r-2})=R \cdot S_j(-\D-\B_{r-2}) \cdot S_{p,q}(\X_2)\,,
$$
where $R=R(\X_2, \D + \B_{r-2})$. Therefore, using Eq.~(\ref{cr}),
we obtain
\begin{equation}\label{sumj}
{\cal T}_r(\X_2-\D-\B_{r-2}) = R \cdot
\Bigl(\sum_{j=0}^{r-1} S_j(-\D-\B_{r-2})
\ \frac{{\overline{\cal T}_{r-j}}(\X_2)}{(x_1x_2)^{r-j+1}}\Bigr)\,.
\end{equation}
By the induction assumption, for positive $j\le r-1$, we have
$$
{\overline{\cal T}_{r-j}}(\X_2)=(x_1x_2)^{r-j+1} \ S_{r-1-j}(\D)\,.
$$
We use now the fact that among the equations characterizing $\cT_r$ is (\ref{III22r})
(because the codimension of $III_{2,2}(r)$ is smaller than $\codim(I_{2,2}(r))$).
Substituting this to (\ref{sumj}), we obtain
\begin{equation}\label{l1}
\sum_{j=1}^{r-1} S_j(-\D-\B_{r-2}) S_{r-1-j}(\D)
+\frac{{\overline{\cal T}_r}(\X_2)}{(x_1x_2)^{r+1}}=0\,.
\end{equation}
But we also have, by a formula for addition of alphabets,
\begin{equation}\label{l2}
\sum_{j=1}^{r-1} S_j(-\D-\B_{r-2}) S_{r-1-j}(\D)+S_{r-1}(\D)
=S_{r-1}(-\B_{r-2})=0\,.
\end{equation}
Combining (\ref{l1}) and (\ref{l2}), gives
$$
{\overline{\cal T}_r}(\X_2)=(x_1x_2)^{r+1} \ S_{r-1}(\D)\,,
$$
that is, the induction assertion.
\qed

\medskip

The Schur function expansion of $S_i(\D)$ was described in \cite{P},
\cite{LLT} and Appendix A3 in \cite{P223}, in the context of the
{\it Segre classes} of the second symmetric power of a rank $2$ vector
bundle. Indeed, $\D$ is the alphabet of the Chern roots of the second
symmetric power of a rank $2$ bundle with the Chern roots $x_1, x_2$.
The recursions encoded by the Pascal diagram (\ref{pstair1}) express
the recursions for the coefficients of the Segre classes of the second 
symmetric power of a rank 2 vector bundle ({\it loc.cit.}).

\subsection{Recursions for $A_3(r)$}

The material of this subsection stems from \cite{PA3}.

We set
\begin{equation}
F_r:=\sum_{j_1\le j_2 \le r} S_{j_1,j_2}(\fbox{$2$}+\fbox{$3$})
S_{r-j_2,r-j_1,r+j_1+j_2}\,.
\end{equation}
This function is the $1$-part of ${\cal T}^{A_3}_r$ (see Section \ref{Schurexp}).

In \cite{Riminvens}, the author gave Thom polynomials for $A_3(1)$ and $A_3(2)$.
Their Schur function expansions are
\begin{equation}
{\cal T}^{A_3}_1=S_{111}+5S_{12}+6S_3=F_1\,.
\end{equation}
and
\begin{equation}
{\cal T}^{A_3}_2 = S_{222}+5S_{123}+6S_{114}+19S_{24}+30S_{15}+36S_6+5S_{33}= F_2 + 5S_{33}\,.
\end{equation}
Note that the $2$-part of ${\cal T}^{A_3}_2$ is $5S_{33}$.

\smallskip

We now pass to the case of general $r$.
Since $A_3(r)$ has codimension $3r$, a partition appearing in the $2$-part of ${\cal T}^{A_3}_r$ has weight $3r$ and its diagram contains the partition $(r+1,r+1)$. Moreover, it can have at most three rows. 

Consider the matrix whose $(i,j)$th entry is the partition $(1+i+j,2+2i-j)$ with the convention that $(1+i+j,2+2i-j)$ is the empty partition for $2j>i+1$: 
$$
\left[\begin{array}{cccccc}
33 & \emptyset & \emptyset & \emptyset & \emptyset & \ldots \\
45 & \emptyset & \emptyset & \emptyset & \emptyset & \ldots \\
57 & 66 & \emptyset & \emptyset & \emptyset & \ldots \\
69 & 78 & \emptyset & \emptyset & \emptyset & \ldots \\
7,11 & 8,10 & 99 & \emptyset & \emptyset & \ldots \\
8,13 & 9,12 & 10,11 & \emptyset & \emptyset & \ldots \\
9,15 & 10,14 & 11,13 & 12,12 & \emptyset & \ldots \\
\vdots & \vdots & \vdots & \vdots & \vdots &
\end{array}\right]
$$
We now want to define a symmetric function $\overline H_r$ whose Schur summands are indexed by partitions from the $(r-1)$th row of
the above matrix. Their coefficients will be given by the corresponding entries of the following Pascal staircase.
Consider the following Taylor expansion:
\[
\aligned
f(z)=\frac{5-6z}{(1-z)(1-2z)(1-3z)}\\
=5+24z+&89z^2+300z^3+965z^4+3024z^5+9329z^6+\ldots\,.
\endaligned
\]
The Pascal staircase associated with $f$ is the following infinite matrix:
\[
P=\left[\begin{array}{cccccc}
5 & 0 & 0 & 0 & 0 & \ \ldots \\
24 & 0 & 0 & 0 & 0 & \ldots \\
89 & 24 & 0 & 0 & 0 & \ldots \\
300 & 113 & 0 & 0 & 0 & \ldots \\
965 & 413 & 113 & 0 & 0 & \ldots \\
3024 & 1378 & 526 & 0 & 0 & \ldots \\
9329 & 4402 & 1904 & 526 & 0 & \ldots \\
\vdots & \vdots & \vdots & \vdots & \vdots &
\end{array}
\right]
\]

For $r\ge 2$, we set
\begin{equation}
{\overline H_r}:=\sum_{2j\leq r} P_{r-1,j} S_{r+j, 2r-j}\,.
\end{equation}

\begin{example}
We have the following values of $\overline H_r$, $r=2,\ldots,7$~:
$$
\aligned
\overline H_2&=5S_{33}\\
\overline H_3&=24S_{45}\\
\overline H_4&=24S_{66}\plus89S_{57}\\
\overline H_5&=113S_{78}\plus300S_{69}\\
\overline H_6&=113S_{99}\plus413S_{8,10}\plus965S_{7,11}\\
\overline H_7&=526S_{10,11}\plus1378S_{9,12}\plus3024S_{8,13}\,.
\endaligned
$$
\end{example}

We define by induction on $r$
$$
H_r={\overline H_r}+\Phi_3({\overline H_{r-1}})+\Phi_3^2({\overline H_{r-2}})
+\cdots+\Phi_3^{r-2}({\overline H_2})\,.
$$
With this definition of $H_r$, we state the following result.

\begin{theorem}\label{a3}(\cite{PA3}) \ We have
$$
\cT^{A_3}_r= F_r + H_r\,.
$$
\end{theorem} 
In other words, the function $H_r$ is the $2$-part of ${\cal T}^{A_3}_r$, and its $h$-parts are zero for $h\ge 3$. Note also that
we recover the recurrence (\ref{FF}):
$$
F_r={\overline F_r}+\Phi_{3}(F_{r-1})\,. 
$$

\bigskip

We show now, following \cite{PA3}, the essential computations in the proof of Theorem \ref{a3}.
As explained in \cite{PA3}, it is crucial to show the vanishing (\ref{III22r}) of ${\cal T}^{A_3}_r$ 
at the Chern class $c(III_{2,2}(r))$. I.e., it suffices to show the equality
\begin{equation}\label{frhr}
(F_r+H_r)(x_1+x_2-\D-\B_{r-2})=0\,.
\end{equation}

Due to the factorization property (\ref{Fact}), each Schur function occuring in the expansion of $H_r$ 
is such that
\begin{equation*}
S_{c,r+1+a,r+1+b}(\X_2\moins \D\moins \B_{r-2})
= R(\X_2,\D\plus \B_{r-2})\cdot S_c(\moins \D\moins \B_{r-2})\cdot S_{a,b}(\X_2)\,.
\end{equation*}
We set
\begin{equation}
V_r(\X_2;\B_{r-2}):= \frac{H_r(\X_2-\D-\B_{r-2})}{R(\X_2,\D+\B_{r-2})}\,,
\end{equation}
so that
\begin{equation}\label{v}
V_r(\X_2;\B_{r-2})=\sum_{i=0}^{r-2} \sum_{\{j\ge 0: \ i+2j \le r-2\}}
e_{r-i,j} \ S_i(-\D-\B_{r-2}) \ S_{j,r-i-j-2}(\X_2)\,.
\end{equation}
We have the following
recursive relation which follows from the observation that the coefficient
of $b_{r-2}$ in $V_r(\X_2;\B_{r-2})$ is equal to
$-V_{r-1}(\X_2;\B_{r-3})$.

\begin{lemma}\label{Lvr} For $r\ge 2$, we have
\begin{equation}
V_r(\X_2; \B_{r-2})= \sum_{i=0}^{r-2} \ V_{r-i}(\X_2; 0)
\ S_i(-\B_{r-2})\,.
\end{equation}
\end{lemma}

Thus it is sufficient to compute $V_r(\X_2;0)$.

\begin{proposition}\label{Pv} For $r\ge 2$, we have
\begin{equation}\label{Vr0}
V_r(\X_2; 0)=
3^{r-2}\Bigl(3 S_{r-2}(\X_2) - 2S_{1,r-3}(\X_2)\Bigr)\,.
\end{equation}
\end{proposition}
(In particular, $V_2(\X_2;0)=5$ and $V_3(\X_2;0)=9S_1(\X_2)$\,.)

We now apply Corollary \ref{pcorollary} from Subsection (\ref{pascal})
with $\B=\fbox{$2x_1$}+\fbox{$2x_2$}$. Expanding
\begin{multline*}
 S_d\left(\X_2 - y^{-1}(\fbox{$ 2x_1$}+\fbox{$ 2x_2$})\right)\\
= S_d(\X_2) - \frac{2x_1+2x_2}{y} S_{d-1}(\X_2)
 + 4\frac{x_1x_2}{y^2} S_{d-2}(\X_2)\,,
\end{multline*}
we get, for $d\ge 3$,
\begin{equation}
W(d,\D) = y^{d-2}(y-1)(y-2) S_{d}(\X_2))-2y^{d-3}(y-1)(y-2) S_{1,d-1}(\X_2)
\end{equation}
and initial conditions
\begin{gather*}
W(0) =1, \quad W(1)= (y-3)S_1(\X_2), \\
W(2)=(y-1)(y-2)S_2(\X_2)-2(y-3)S_{11}(\X_2)\,.
\end{gather*}

We come back to Proposition \ref{Pv}, and we take the Pascal staircase (\ref{pstair1}).
Then for $d=r-2$, the function $W(d,\D)$ is the function $V_r(\X_2;0)$. We thus have to 
specialize $y$ into $1,2,3$ successively.
Apart from initial values, only $y=3$ contributes, 
and we get, for $d\ge 3$,
$$ 
W(d,\D)= 3^{d+1}S_n(\X_2)- 2\cdot 3^d S_{1,d-1}(\X_2)\,.
$$

This proves Proposition \ref{Pv}, checking the cases $r=2,3,4$ directly.

\medskip

We now pass to the specialization 
$F_r(\X_2- \D-\B_{r-2})$. It is rather straightforward to prove the following
lemma (cf. \cite{PA3}).

\begin{lemma}\label{div} The resultant $R(\X_2,\D+\B_{r-2})$ divides 
$F_r(\X_2- \D- \B_{r-2})$.
\end{lemma}

We set
\begin{equation}
U_r(\X_2;\B_{r-2}):=\frac{F_r(\X_2\moins \D\moins \B_{r-2})}{R(\X_2,\D+\B_{r-2})}\,.
\end{equation} 

Note that each variable
$b\in \B_{r-2}$ appears at most with degree $3$ in\break
$F_r(\X_2\moins \D\moins \B_{r-2})$, and hence at most
with degree $1$ in $U_r(\X_2;\B_{r-2})$. We have the following precise
recursive relation which follows from the observation that
the coefficient of $b_{r-2}^3$ in
$F_r(\X_2\moins \D\moins \B_{r-2})$
is equal to $F_{r-1}(\X_2\moins \D\moins \B_{r-3})$.

\begin{lemma}\label{Lur} For $r\ge 2$, we have
\begin{equation}
U_r(\X_2; \B_{r-2})= \sum_{i=0}^{r-2} \ U_{r-i}(\X_2; 0)
\ S_i(-\B_{r-2})\,.
\end{equation}
\end{lemma}

Let $\pi$ be the endomorphism of the $\bf C$-vector space of functions 
of $x_1,x_2$, defined by
$$
\pi \bigl(f(x_1,x_2)\bigr) =
\frac{x_1 f(x_1,x_2)-x_2 f(x_2,x_1)}{x_1-x_2}\,.
$$
For any $i,j\in {\bf N}$, we have
\begin{equation}\label{pi}
\pi(x_1^jx_2^i)=S_{i,j}(\X_2)\,.
\end{equation}

The proof of the following proposition will make use of multi-Schur functions
(see the end of Section \ref{Schur}).

\begin{proposition}\label{gr} For $r\ge 2$, we have
\begin{equation}\label{ilo} 
F_r(\X_2-\D)=-3^{r-2}R(\X_2,\D) (x_1x_2)^{r-2}\bigl(3S_{r-2}(\X_2)-2S_{1,r-3}(\X_2)\bigr)\,.
\end{equation}
\end{proposition}
\begin{proof}
The identity is true for $r=2$.
To prove the assertion for $r\ge 3$, 
we compute in two different ways the action of $\pi$ on the multi-Schur function
\begin{equation}\label{multi}
S_{r,r;r}(\X_2+\fbox{$2x_1$}+\fbox{$3x_1$}-\D;x_1-\D)\,.
\end{equation}

Firstly, expanding (\ref{multi}), we have
$$
\aligned
&\pi \bigl(S_{r,r;r}(\X_2+\fbox{$2x_1$}+\fbox{$3x_1$}-\D;x_1-\D)\bigr)\cr
&=\pi \bigl(\sum_{j_1\le j_2 \le r} S_{j_1,j_2}(\fbox{$2x_1$}+\fbox{$3x_1$}) \ 
S_{r-j_2,r-j_1;r}(\X_2-\D;x_1-\D)\bigr)\cr
&=\pi \bigl(\sum_{j_1\le j_2 \le r} S_{j_1,j_2}(\fbox{$2$}+\fbox{$3$}) \ 
S_{r-j_2,r-j_1;r+j_1+j_2}(\X_2-\D;x_1-\D)\bigr)\cr
&=\sum_{j_1\le j_2 \le r} S_{j_1,j_2}(\fbox{$2$}+\fbox{$3$}) \ 
S_{r-j_2,r-j_1,r+j_1+j_2}(\X_2-\D)\cr
&=F_r(\X_2-\D)\,.
\endaligned
$$

Secondly, using Lemma \ref{TL}, we subtract $x_1$ from the arguments in the first two rows of the determinant (\ref{multi}) 
without changing its value. We get the determinant

\smallskip

$$
\small
\begin{vmatrix}
S_r(x_2+\fbox{$2x_1$}+\fbox{$3x_1$}-\D)  & S_{r+1}(x_2+\fbox{$2x_1$}+\fbox{$3x_1$}-\D)& S_{r+2}(-\D)\\
S_{r-1}(x_2+\fbox{$2x_1$}+\fbox{$3x_1$}-\D) & S_r(x_2+\fbox{$2x_1$}+\fbox{$3x_1$}-\D)& S_{r+1}(-\D) \\
S_{r-2}(\X_2+\fbox{$2x_1$}+\fbox{$3x_1$}-\D)& S_{r-1}(\X_2+\fbox{$2x_1$}+\fbox{$3x_1$}-\D)& S_r(x_1-\D) 
\end{vmatrix}\,.
$$

\bigskip

\noindent
Since the elements in the first two rows of the third column are zero, this determinant is equal to
$$
S_{r,r}(x_2+\fbox{$2x_1$}+\fbox{$3x_1$}-\D) \cdot S_r(x_1-\D)\,.
$$
Since
$$
x_2+\fbox{$2x_1$}+\fbox{$3x_1$}-\D=x_2+\fbox{$3x_1$}-\fbox{$2x_2$}-\fbox{$x_1+x_2$}
$$
and the following two factorizations hold:
$$
S_{r,r}(x_2+\fbox{$3x_1$}-\fbox{$2x_2$}-\fbox{$x_1+x_2$})= 
-3^{r-2}(x_2-2x_1)(x_1x_2)^{r-1}(3x_1-2x_2)
$$
and
$$
S_r(x_1-\D)=x_1^{r-2}x_2(x_1-2x_2)\,,
$$
we infer that
\begin{multline}\label{factor}
S_{r,r;r}(\X_2+\fbox{$2x_1$}+\fbox{$3x_1$}-\D;x_1-\D)\\{}=
-3^{r-2}R(\X_2,\D)(x_1x_2)^{r-2}x_1^{r-3}(3x_1-2x_2)\,.
\end{multline}
By (\ref{pi}), the result of applying $\pi$ to (\ref{factor}) is
$$
-3^{r-2}R(\X_2,\D) (x_1x_2)^{r-2}\bigl(3S_{r-2}(\X_2)-2S_{1,r-3}(\X_2)\bigr)\,.
$$

Comparison of both computations of $\pi$ applied to (\ref{multi}) 
yields the proposition.
\end{proof}

In terms of $U_r$, we rewrite Proposition \ref{gr} into 

\begin{corollary}\label{Pu} For $r\ge 2$, we have
\begin{equation}\label{Ur0}
U_r(\X_2;0)=-3^{r-2}\bigl(3 S_{r-2}(\X_2)-2S_{1,r-3}(\X_2)\bigr)\,.
\end{equation}
\end{corollary}

These are the essential computations with Schur functions leading to the proof of Theorem \ref{a3}.

\section{Towards the Thom polynomial of $III_{3,3}(r)$}\label{III33}

The singularity
$III_{3,3}(r)$ has codimension $4r+2$. So, the partitions that we need to consider have weight $4r+2$. Moreover, all diagrams contain the partition $(r+1,r+1)$, have at most 4 rows and the length of the second row is at most $r$. Let 
$\mathbf{D}_r$ denote the set of all such diagrams. By $\mathbf{D}_{r,2}, \mathbf{D}_{r,3} and \mathbf{D}_{r,4}$ we shall denote the subsets of $\mathbf{D}_r$, that consist of diagrams with 2,3 and 4 rows, respectively.

Set ${\mathcal T}_r:=\cT^{III_{3,3}}_r$.
Then, the part of ${\mathcal T}_r$ corresponding to the partitions in $\mathbf{D}_{r,4}$ is given by $\Phi_{4}({\mathcal T}_{r-1})$.

The Thom polynomial ${\mathcal T}_r$ must satisfy the following system of equations: (\ref{air}) for $i=0,1,2,3,4$\,, (\ref{I22r}),
(\ref{I23r}), (\ref{III22r}), (\ref{III23r}), (\ref{III24r}) together with the normalizing equation (\ref{nIII33r}).

For a partition $I \in \mathbf{D}_r$, we have
\begin{equation*}
\begin{split}
S_I(-\B_{r-1})= & S_I(x-\B_{r-1}-\fbox{$2x$})\\
= & S_I(x\moins \B_{r-1}-\fbox{$3x$}) \\
= & S_I(x\moins \B_{r-1}-\fbox{$4x$})=0\, .
\end{split}
\end{equation*}
Hence Eqs. (\ref{air}) for $i=0,1,2,3,4$ are satisfied automatically by any linear combination of Schur functions indexed by partitions in $\mathbf{D}_r$. Moreover, Eq.(\ref{I22r}) implies Eq.(\ref{III22r}) by the substitution $b_{r-1}=\fbox{$x_1+x_2$}$. 
Hence we can replace the former set of equations by a smaller set of equations consisting of Eqs. : (\ref{I22r}), (\ref{I23r}), (\ref{III23r}),  (\ref{III24r}) and (\ref{nIII33r}). Note that in these equations, the alphabets we need to consider, are suitable for the factorization property (\ref{Fact}) associated with a pair of alphabets of cardinalities $r+1$ and $2$.

In \cite{algorithm}, we give an algorithm based on ACE (cf. \cite{V}) which solves the latter system of equations.
Using this algorithm, we get the (unique) $\cT_r$ for $r=2,\ldots, 8$, expanded in the Schur function basis. In the next example, we give $\cT_r$ for $r=2,3$, and in Section \ref{appendix1}, we give $\cT_4,\ldots,\cT_8$.

\begin{example} We have
\begin{equation*}
\begin{split}
{\mathcal T}_2  & = 4S_{37} + 16 S_{46} + 28 S_{55}   \\
 + &  20 S_{145} + 6 S_{136} + 7 S_{235} + 3 S_{244} \\
+ & 2 S_{1135} + 3 S_{1234}+ 6 S_{1144}+ S_{2233}\,; \\
\end{split}
\end{equation*}

\begin{equation*}
\begin{split}
{\mathcal T}_3  & = 8S_{4,10} +40 S_{59} +88S_{68} +120S_{77}   \\
& + 12S_{149} +52S_{158} +100S_{167}  \\
& +14 S_{248} +50S_{257} +20S_{266} \\
& +15S_{347} +10S_{356} + \Phi_4({\mathcal T}_2)\,.\\
\end{split}
\end{equation*}

\end{example}

In \cite{O3}, the author proposes a conjecture about the recursion for the coefficients in the Schur function 
expansion of $\cT_r$. This recursion is checked for $2\le r\le 8$, with the help of an algorithm in \cite{algorithm}.

\section{On the Thom polynomial of $I_{2,3}(r)$}\label{I23}

Set ${\cal T}_r:={\cal T}_r^{I_{2,3}}$. 
The singularity
$I_{2,3}(r)$ has codimension $4r+1$. So, the partitions that we need to consider have weight $4r+1$. Moreover, all diagrams contain the partition $(r+1,r+1)$ and have at most 4 rows. 
Then, the part of ${\cT}_r$ corresponding to the partitions with 4 rows is given by $\Phi_4({\mathcal T}_{r-1})$.

The Thom polynomial ${\cT}_r$ must satisfy the following system of equations: (\ref{air}) for $i=0,1,2,3$\,, (\ref{I22r}),
(\ref{III22r}), (\ref{III23r}) together with the normalizing equation (\ref{nI23r}).

An algorithm analogous to the one in \cite{algorithm}, allows us to get the (unique) solutions $\cT_r$ of this system of equations
for $r=1,\ldots, 7$, expanded in Schur function basis. In the next example, we give $\cT_r$ for $r=1,2,3$, and in Section \ref{appendix2}, we give $\cT_4,\ldots,\cT_7$.

\begin{example} We have
\begin{equation*}
{\mathcal T}_1 = 2 S_{122} + 4 S_{23};
\end{equation*}

\begin{equation*}
 {\mathcal T}_2 = 32 S_{36} + 24 S_{45} + 24 S_{135} + 12 S_{144} + 12 S_{234} +3  S_{333}+\Phi_4({\mathcal T}_1)\,; 
\end{equation*}

\begin{equation*}
\begin{split}
 {\mathcal T}_3  & =
   208 S_{49} 
 + 208 S_{58} 
 + 112 S_{67}\\
& + 168 S_{148}
 + 152 S_{157}
 + 56 S_{166}\\
 & + 100 S_{247}
 + 76 S_{256} \\
& + 50 S_{346}
 + 24 S_{355}\\
& + 18 S_{445}
+\Phi_4({\mathcal T}_2)\,.
\end{split}
\end{equation*}

\end{example}

\section{Appendix 1: $\cT^{III_{3,3}}_r$, $r=4,\ldots, 8$}\label{appendix1}
\small

Let ${\mathcal T}_r=\cT^{III_{3,3}}_r$. We have

\begin{equation*}
\begin{split}
{\mathcal T}_4  & =  16S_{5,13} + 96S_{6,12} + 256S_{7,11} + 416S_{8,10} + 496 S_{9,9}   \\
& +24 S_{1,5,12} +128 S_{1,6,11} +304 S_{1,7,10} +448 S_{189} \\
& +28 S_{2,5,11} +128 S_{2,6,10} + 264S_{279} +100 S_{288} \\
& + 30S_{3,5,10} + 112S_{369} +70 S_{378} \\
& +31 S_{459} +25 S_{468} +10S_{477} +\Phi_4({\mathcal T}_3)\,;\\
\end{split}
\end{equation*}

\begin{equation*}
\begin{split}
{\mathcal T}_5  & =  32S_{6,16}+224S_{7,15}+704S_{8,14}+1344S_{9,13 }+1824S_{10,12}+2016S_{11,11 }\\
& + 48 S_{1,6,15 }+304S_{1,7,14 }+864S_{1,8,13 }+1504S_{1,9,12}+1904S_{1,10,11 } \\
& + 56S_{2,6,14 }+312S_{2,7,13 }+784S_{2,8,12 }+1232S_{2,9,11} +448S_{2,10,10} \\
& + 60 S_{3,6,13 }+284S_{3,7,12 }+616S_{3,8,11 }+364S_{3,9,10 }\\
& + 62S_{4,6,12 }+238S_{4,7,11 }+182S_{4,8,10 }+70S_{499 }  \\
& + 63 S_{5,6,11 }+56S_{5,7,10 }+35S_{589 } +\Phi_4({\mathcal T}_4)\,; \\
\end{split}
\end{equation*}

\begin{equation*}
\begin{split}
{\mathcal T}_6 &= 64S_{7,19}+512S_{8,18}+1856S_{9,17}+4096S_{10,16}+6336S_{11,15}+7680S_{12,14}+8128S_{13,13} \\
&+96 S_{1,7,18}+704S_{1,8,17}+2336S_{1,9,16}+4736S_{1,10,15}+6816S_{1,11,14}+7872S_{1,12,13} \\
&+112S_{2,7,17}+736 S_{2,8,16}+2192S_{2,9,15} +4032 S_{2,10,14}+5392S_{2,11,12}+1904S_{2,12,12}\\ 
&+120S_{3,7,16}+688S_{3,8,15} +1800 S_{3,9,14} +2976S_{3,10,13}+1680S_{3,11,12}\\  
& +124 S_{4,7,15}+600S_{4,8,14}+1348S_{4,9,13}+980S_{4,10,12}+364S_{4,11,11}  \\
&+ 126 S_{5,7,14} +492S_{5,8,13} +420S_{5,9,12}+252S_{5,10,11}\\
&+127S_{6,7,13} +119S_{6,8,12}+91S_{6,9,11}+35S_{6,10,10}+ \Phi_4({\mathcal T}_5)\,; \\
\end{split}
\end{equation*}

\vspace{0.3cm}
\begin{equation*}
\begin{split}
{\mathcal T}_7  & =128 S_{8,22}+1152S_{9,21}+4736S_{10,20}+11904S_{11,19}+20864S_{12,18}\\
&+28032S_{13,17}+31616S_{14,16}+32640S_{15,15}\\
&+192S_{1,8,21}+1600S_{1,9,20}+6080S_{1,10,19}+14144S_{1,11,18}+23104S_{1,12,17}\\
&+29376S_{1,13,16}+32064S_{1,14,15}\\ 
&+224S_{2,8,20}+1696S_{2,9,19}+5856S_{2,10,18}+12448S_{2,11,17}+18848S_{2,12,16}\\
&+22752S_{2,13,15}+7872S_{2,14,14}\\ 
&+240S_{3,8,19}+1616S_{3,9,18}+4976S_{3,10,17}+9552S_{3,11,16}+13392S_{3,12,15}+7296S_{3,13,14}\\ 
&+248S_{4,8,18}+1448S_{4,9,17}+3896S_{4,10,16}+6696S_{4,11,15}+4656S_{4,12,14}+1680S_{4,13,13}\\ 
&+252S_{5,8,17}+1236S_{5,9,16}+2844S_{5,10,15}+2328S_{5,11,14}+1344S_{5,12,13}\\ 
&+254S_{6,8,16}+1002S_{6,9,15}+912S_{6,10,14}+672S_{6,11,13}+252S_{6,12,12}\\
&+255S_{7,8,15}+246S_{7,9,14}+210S_{7,10,13}+126S_{7,11,12}+\Phi_4({\mathcal T}_6)\,; \\
\end{split}
\end{equation*}

\vspace{-0.15cm}

\begin{equation*}
\begin{split}
{\mathcal T}_8  &=256S_{9,25}+2560S_{10,24}+11776S_{11,23}+33280S_{12,22}+65536S_{13,21}+97792S_{14,20}\\
&+119296S_{15,19}+128512S_{16,18}+130816S_{17,17}\\
&+384S_{1,9,24}+3584S_{1,10,23}+15360S_{1,11,22}+40448S_{1,12,21}+74496S_{1,13,20}\\
&+104960S_{1,14,19}+122880S_{1,15,18}+129536S_{1,16,17}\\
&+448S_{2,9,23}+3840S_{2,10,22}+15104S_{2,11,21}+36608S_{2,12,20}+62592S_{2,13,19}\\
&+83200S_{2,14,18}+93952S_{2,15,17}+32064S_{2,16,16}\\
&+480S_{3,9,22}+3712S_{3,10,21}+13184S_{3,11,20}+29056S_{3,12,19}\\
&+45888S_{3,13,18}+57728S_{3,14,17}+30624S_{3,15,16}\\
&+496S_{4,9,20}+3392S_{4,10,20}+10688S_{4,11,19}+21184S_{4,12,18}\\
&+30880S_{4,13,17}+20688S_{4,14,16}+7296S_{4,15,15}\\
&+504S_{5,9,20}+2976S_{5,10,19}+8160S_{5,11,18}+14432S_{5,12,17}+11352S_{5,13,16}+6336S_{5,14,15}\\
&+508S_{6,9,19}+2512S_{6,10,18}+5872S_{6,11,17}+5172S_{6,12,16}+3672S_{6,13,15}+1344S_{6,14,14}\\
&+510S_{7,9,18}+2024S_{7,10,17}+1914S_{7,11,16}+1584S_{7,12,15}+924S_{7,13,14}\\
&+511S_{8,9,17}+501S_{8,10,16}+456S_{8,11,15}+336S_{8,12,14}+126S_{8,13,13}
+\Phi_4({\mathcal T}_7)\,. \\
\end{split}
\end{equation*}

\section{Appendix 2: $\cT^{I_{2,3}}_r$, $r=4,\ldots, 7$}\label{appendix2}
\small

Let ${\mathcal T}_r=\cT^{I_{2,3}}_r$. We have

\begin{equation*}
 \begin{split}
   {\mathcal T}_4  & =
 1280 S_{5,12}
 + 1024 S_{7,10}
 + 1408 S_{6,11}
 + 480 S_{89}\\
& + 1056 S_{1,5,11}
 + 1120 S_{1,6,10}
 + 736 S_{179}
 + 240 S_{188}\\
& + 656 S_{2,5,10}
 + 656 S_{269}
 + 368 S_{278}\\
& + 360 S_{359}
 + 328 S_{368}
 + 124 S_{377}\\
& + 180 S_{458}
 + 134 S_{467}\\
& + 75 S_{557}
 + 36 S_{566}
+\Phi_4({\mathcal T}_3)\,; 
 \end{split}
\end{equation*}

\begin{equation*}
 \begin{split}
{\mathcal T}_5  & =
7744 S_{6,15}
 + 8832 S_{7,14}
 + 7168 S_{8,13}
 + 4544 S_{9,12}
 + 1984 S_{10,11}\\
& + 6432 S_{1,6,14}
 + 7232 S_{1,7,13}
 + 5632 S_{1,8,12}
 + 3232 S_{1,9,11}
 + 992 S_{1,10,10}\\
& + 4048 S_{2,6,13}
 + 4448 S_{2,7,12}
 + 3264 S_{2,8,11}
 + 1616 S_{2,9,10}\\
& + 2280 S_{3,6,12}
 + 2416 S_{3,7,11}
 + 1632 S_{3,8,10}
 + 560 S_{3,9,9}\\
& + 1204 S_{4,6,11}
 + 1208 S_{4,7,10}
 + 692 S_{489}\\
& + 602 S_{5,6,10}
 + 542 S_{579}
 + 206 S_{588}\\
&
 + 270 S_{669}
 + 201 S_{678}
+\Phi_4({\mathcal T}_4)\,;
 \end{split}
\end{equation*}

\begin{equation*}
 \begin{split}
  {\mathcal T}_6  & =
46592 S_{7,18}
 + 53888 S_{8,17}
 + 45824 S_{9,16}
 + 32640 S_{10,15}
 + 19200 S_{11,14}
 + 8064 S_{12,13}
\\
& + 38784 S_{1,7,17}
 + 44608 S_{1,8,16}
 + 37248 S_{1,9,15}
 + 25408 S_{1,10,14} 
 + 13568 S_{1,11,13}
 + 4032 S_{1,12,12}
\\
& + 24512 S_{2,7,16}
 + 27936 S_{2,8,15}
 + 22720 S_{2,9,14}
 + 14624 S_{2,10,13}
 + 6784 S_{2,11,12}
\\
& + 13920 S_{3,7,15}
 + 15632 S_{3,8,14}
 + 12256 S_{3,9,13}
 + 7312 S_{3,10,12}
 + 2384 S_{3,11,11}
\\ 
& + 7472 S_{4,7,14}
 + 8200 S_{4,8,13}
 + 6128 S_{4,9,12}
 + 3152 S_{4,10,11}
\\
& + 3864 S_{5,7,13}
 + 4100 S_{5,8,12}
 + 2812 S_{5,9,11}
 + 980 S_{5,10,10}
\\
& + 1932 S_{6,7,12}
 + 1924 S_{6,8,11}
 + 1108 S_{6,9,10}
\\
& + 903 S_{7,7,11}
 + 813 S_{7,8,10}
 + 309 S_{7,9,9}
+\Phi_4({\mathcal T}_5)\,;
 \end{split}
\end{equation*}

\begin{equation*}
 \begin{split}
    {\mathcal T}_7  & =
279808 S_{8,21}
 + 325376 S_{9,20}
 + 282368 S_{10,19}
 + 212224 S_{11,18}\\
 &+ 140544 S_{12,17}
 + 79104 S_{13,16}
 + 32512 S_{14,15}
\\
& + 233088 S_{1,8,20}
 + 270464 S_{1,9,19}
 + 232832 S_{1,10,18}
 + 171392 S_{1,11,17}\\
 &+ 108672 S_{1,12,16}
 + 55680 S_{1,13,15}
 + 16256 S_{1,14,14}
\\
& + 147520 S_{2,8,19}
 + 170560 S_{2,9,18}
 + 145088 S_{2,10,17}\\
&+ 103872 S_{2,11,16}
 + 62272 S_{2,12,15}
 + 27840 S_{2,13,14}
\\
& + 84000 S_{3,8,18}
 + 96544 S_{3,9,17}
 + 80736 S_{3,10,16}\\
 &+ 55776 S_{3,11,15}
 + 31136 S_{3,12,14}
 + 9856 S_{3,13,13}
\\ 
& + 45328 S_{4,8,17}
 + 51600 S_{4,9,16}
 + 42160 S_{4,10,15}
 + 27888 S_{4,11,14}
 + 13536 S_{4,12,13}
\\ 
& + 23688 S_{5,8,16}
 + 26568 S_{5,9,15}
 + 21080 S_{5,10,14}
 + 12928 S_{5,11,13}
 + 4304 S_{5,12,12}
\\
& + 12100 S_{6,8,15}
 + 13284 S_{6,9,14}
 + 10032 S_{6,10,13}
 + 5232 S_{6,11,12}
\\ 
& + 6050 S_{7,8,14}
 + 6388 S_{7,9,13}
 + 4400 S_{7,10,12} 
 + 1540 S_{7,11,11}
\\
& + 2898 S_{8,8,13}
 + 2886 S_{8,9,12}
 + 1662 S_{8,10,11}
+\Phi_4({\mathcal T}_6)\,.
 \end{split}
\end{equation*}


\begin{thebibliography}{99}
%\small

\bibitem{AVGL} V. Arnold, V. Vasilev, V. Goryunov, O. Lyashko,
\emph{Singularities. Local and global theory,}
Enc. Math. Sci. vol. 6 (Dynamical Systems VI), Springer, 1993.

\bibitem{B} G. Berczi,
\emph{Moduli of map germs, Thom polynomials and the Green-Griffiths conjecture,}
this volume.

\bibitem{BR} A. Berele, A. Regev,
\emph{Hook Young diagrams with applications to combinatorics and to
representation theory of Lie superalgebras,}
Adv. in Math. {\bf 64} (1987), 118--175.

\bibitem{dPW} A. Du Plessis, C.T.C. Wall,
\emph{The geometry of topological stability,}
Oxford Math. Monographs, 1995.

\bibitem{FK} L. Feh\'er, B. Komuves,
\emph{On second order Thom-Boardman singularities,}
Fund. Math. {\bf 191} (2006), 249--264.

\bibitem{FeRim1} L. Feh\'er, R. Rim\'anyi,
\emph{Calculation of Thom polynomials and other cohomological obstructions
for group actions,} in: ``Real and complex singularities (San Carlos
2002)'', T. Gaffney and M. Ruas (eds.), Contemporary Math. {\bf 354},
2004, 69--93.


\bibitem{FeRim} L. Feh\'er, R. Rim\'anyi,
\emph{Thom series of contact singularities,}
math.AG/0809.2925v2, to appear in Annals of Math.

\bibitem{F} W. Fulton,
\emph{Young tableaux with application to representation theory and geometry,}
Cambridge University Press, 1997.

\bibitem{FL} W. Fulton, R. Lazarsfeld,
\emph{Positive polynomials for ample vector bundles},
Ann. Math. {\bf 118} (1983), 35--60.


\bibitem{FP} W. Fulton, P. Pragacz,
\emph{Schubert varieties and degeneracy loci},
Springer LNM {\bf 1689}, 1998.

\bibitem{G} T. Gaffney,
\emph{The Thom polynomial of $\overline{\Sigma^{1111}}$,}
in: ``Singularities'', Proc. Symposia in Pure Math. {\bf 40(1)},
399--408, AMS, 1983.

\bibitem{Ka} M. \'E. Kazarian,
\emph{On the positivity of Schur expansions
of Thom polynomials (after M. Mikosz, P. Pragacz and A. Weber)},
private communication (03.04.2009).

\bibitem{Kl0} S. Kleiman,
\emph{The transversality of a general translate,}
Comp. Math. {\bf 38} (1974), 287--297.


\bibitem{Kl} S. Kleiman,
\emph{The enumerative theory of singularities,}
in: ``Real and complex singularities, Oslo 1976'', P. Holm (ed.), 1978,
297--396.

\bibitem{LLT} D. Laksov, A. Lascoux, A. Thorup,
\emph{On Giambelli's theorem for complete correlations,}
Acta Math. {\bf 162} (1989), 143--199.

\bibitem{L1} A. Lascoux,
\emph{Fonctions de Schur et grassmanniennes},
C. R. Acad. Sci. {\bf 281} (1975), 813--815.

\bibitem{L} A. Lascoux,
\emph{Symmetric functions and combinatorial operators on polynomials},
CBMS/AMS Lectures Notes  99, Providence, 2003.

\bibitem{L2} A. Lascoux,
\emph{Addition of $\pm 1$: application to arithmetic},
S\'eminaire Lotharingien de Combinatoire, {\bf B52a} (2004), 9 pp.

\bibitem{PA3} A. Lascoux, P. Pragacz,
\emph{Thom polynomials and Schur functions: the singularities
$A_3(-)$,}
Publ. RIMS Kyoto Univ. {\bf 46} (2010), 183-200.

\bibitem{M} I.G. Macdonald,
\emph{Symmetric functions and Hall-Littlewood polynomials,}
Oxford Math. Monographs, Second Edition, 1995.

\bibitem{MPW1} M. Mikosz, P. Pragacz, A. Weber,
\emph{Positivity of Thom polynomials II; the Lagrange singularities,}
Fund. Math. {\bf 202} (2009), 65--79. 

\bibitem{MPW} M. Mikosz, P. Pragacz, A. Weber,
\emph{Positivity of Legendrian Thom polynomials,}
J. Differential Geom. {\bf 89(1)} (2011), 111-132.

\bibitem{O1} \"O. \"Ozt\"urk,
\emph{Thom polynomials and Schur functions: the singularities $A_4(-)$,}
Serdica Math. J. {\bf 33} (2007), 301--320.

\bibitem{O2} \"O. \"Ozt\"urk,
\emph{Thom polynomials and Schur functions: the singularities $III_{2,3}$,}
Ann. Polon. Math. {\bf 99} (2010), 295--304.

\bibitem{O3} \"O. \"Ozt\"urk,
\emph{Ph.D. Thesis,}
IMPAN, Warsaw, 2010.

\bibitem{algorithm} \"O. \"Ozt\"urk,
\emph{Addendum: ACE algorithms for Thom polynomials of $III_{3,3}(r)$,}
http://www.impan.pl/$\sim$pragacz/download/algIII33.pdf

\bibitem{Po} I. Porteous,
\emph{Simple singularities of maps,}
in: ``Proc. Liverpool Singularities I'', Springer LNM {\bf 192}, 1971,
286--307.

\bibitem{P} P. Pragacz,
\emph{Enumerative geometry of degeneracy loci,}
Ann. Sc. Ec. Norm. Sup. {\bf 21} (1988), 413--454.

\bibitem{P2} P. Pragacz,
\emph{Algebro-geometric applications of Schur $S$- and $Q$-polyno\-mials,}
in: ``Topics in invariant theory'', S\'eminaire d'Alg\`ebre
Dubreil-Malliavin 1989-1990, M-P. Malliavin (ed.), Springer LNM {\bf 1478},
1991, 130--191.

\bibitem{P223} P. Pragacz,
\emph{Symmetric polynomials and divided differences in formulas
of intersection theory,}
in: ``Parameter spaces'', P. Pragacz (ed.), Banach Center Publications
{\bf 36}, 1996, 125--177.

\bibitem{PArx} P. Pragacz,
\emph{Thom polynomials and Schur functions I,}
arXiv: math.AG/0509234.

\bibitem{PI22} P. Pragacz,
\emph{Thom polynomials and Schur functions: the singularities
$I_{2,2}(-)$,}
Ann. Inst. Fourier {\bf 57} (2007), 1487--1508. 

\bibitem{PAi} P. Pragacz,
\emph{Thom polynomials and Schur functions: towards the singularities $A_i(-)$,}
in ``Real and complex singularities (Sao Carlos 2006)'', M. J. Saja and J. Seade (eds.), 
Contemporary Mathematics  {\bf 459}, 2008, 165--178.

\bibitem{PT} P. Pragacz, A. Thorup,
\emph{On a Jacobi-Trudi identity for supersymmetric polynomials,}
Adv. in Math. {\bf 95} (1992), 8--17.

\bibitem{PW} P. Pragacz, A. Weber,
\emph{Positivity of Schur function expansions of Thom polynomials,}
Fund. Math.  {\bf 195} (2007), 85--95.

\bibitem{PW2} P. Pragacz, A. Weber,
\emph{Thom polynomials of invariant cones, Schur functions and
positivity,} in: ``Algebraic cycles, sheaves, shtukas, and moduli'', P. Pragacz (ed.), Trends
in Mathematics, Birkh\"auser, 2007, 117--129.


\bibitem{Riminvens} R. Rim\'anyi,
\emph{Thom polynomials, symmetries and incidences of singularities,}
Inv. Math.  {\bf 143} (2001), 499--521.

\bibitem{Ro} F. Ronga,
\emph{Le calcul des classes duales aux singulariti\'es de Boardman
d'ordre $2$,}
Comm. Math. Helv. {\bf 47} (1972), 15--35.

\bibitem{S} J. Stembridge,
\emph{A characterization of supersymmetric polynomials,}
J. Algebra {\bf 95} (1985), 439--487.

\bibitem{T} R. Thom,
\emph{Les singularit\'es des applications diff\'erentiables,}
Ann. Inst. Fourier {\bf 6} (1955--56), 43--87.


\bibitem{V} S. Veigneau,
\emph{ACE, an algebraic combinatorics environment for the computer
algebra system MAPLE,} 1998.


\end{thebibliography}
\end{document}